\documentclass[sn-mathphys,Numbered]{sn-jnl}

\usepackage{graphicx}%
\usepackage{multirow}%
\usepackage{amsmath,amssymb,amsfonts}%
\usepackage{amsthm}%
\usepackage{mathrsfs}%
\usepackage[title]{appendix}%
\usepackage{xcolor}%
\usepackage{textcomp}%
\usepackage{manyfoot}%
\usepackage{booktabs}%
\usepackage{algorithm}%
\usepackage{algorithmicx}%
\usepackage{algpseudocode}%
\usepackage{listings}%

\usepackage{stmaryrd}
\usepackage{colortbl}
\usepackage{autonum}

\theoremstyle{thmstyleone}%
\newtheorem{theorem}{Theorem}
\newtheorem{proposition}[theorem]{Proposition}%

\theoremstyle{thmstyletwo}%
\newtheorem{example}{Example}%

\theoremstyle{thmstylethree}%

\newcommand{\oeis}[1]{\color{blue}\textnormal{\href{https://oeis.org/#1}{#1}}\color{black}}
\newcommand{\na}{\cellcolor{green!5} \color{black}\substack{..\\ \smile}\color{black}}
\newcommand{\conj}{\color{red}\bullet\color{black}}
\newcommand{\noa}{\color{blue}\circ\color{black}}

\begin{document}

\title[Trees with flowers.]{Trees with flowers: A catalog of integer partition and integer composition trees with their asymptotic analysis}


\author[1]{\fnm{Ricardo} \sur{G\'omez A\'iza}} \email{rgomez@math.unam.mx}

\affil[1]{\orgdiv{Instituto de Matem\'aticas}, \orgname{Universidad Nacional Aut\'onoma de M\'exico}, \orgaddress{\street{\'Area de la Investigaci\'on Cient\'ifica, Circuito Exterior, Ciudad Universitaria}, \city{M\'exico}, \postcode{04510}, \state{CDMX}, \country{M\'exico}}\\ \\ {\emph{\footnotesize On the occasion of the 80th anniversay of the Instituto de Matemáticas, UNAM}}}



\abstract{We present families of combinatorial classes described
as trees with nodes that can carry one of two types of ``flowers'':
integer partitions or integer compositions.
Two parameters on the flowers of trees will be considered: the number
of ``petals'' in all the flowers (petals' weight) and the number of edges
in the petals of all the flowers (flowers' weight). 
We give explicit expressions of their generating functions
and deduce general formulas for the asymptotic growth
of their coefficients and the expectations
of their concentrated distributions.
}




\maketitle

\section{Introduction}\label{sec:introduction}

A {\em tree with flowers} is a rooted tree together with some edge-disjoint cycles called {\em petals} attached to some of 
the nodes of the trees. These petals are vertex-disjoint as well, except for the vertex of the tree to which they are
attached. The collection of petals attached to a given vertex is called a {\em flower}. 
The size of a tree with  flowers is its {\em total} number of edges, that is, the number of edges of the tree
(the \emph{tree's weight}) plus the 
number of edges in all the petals of all the flowers in the tree
(the \emph{flowers' weight}).
The \emph{petals' weight} of a tree with flowers is the
total number of petals in all the flowers of the tree. 
See Figure \ref{fig:tree}.

\begin{figure}[h] 
   \centering
   $\underset{\textnormal{(a)}}{\includegraphics[width=1.58in]{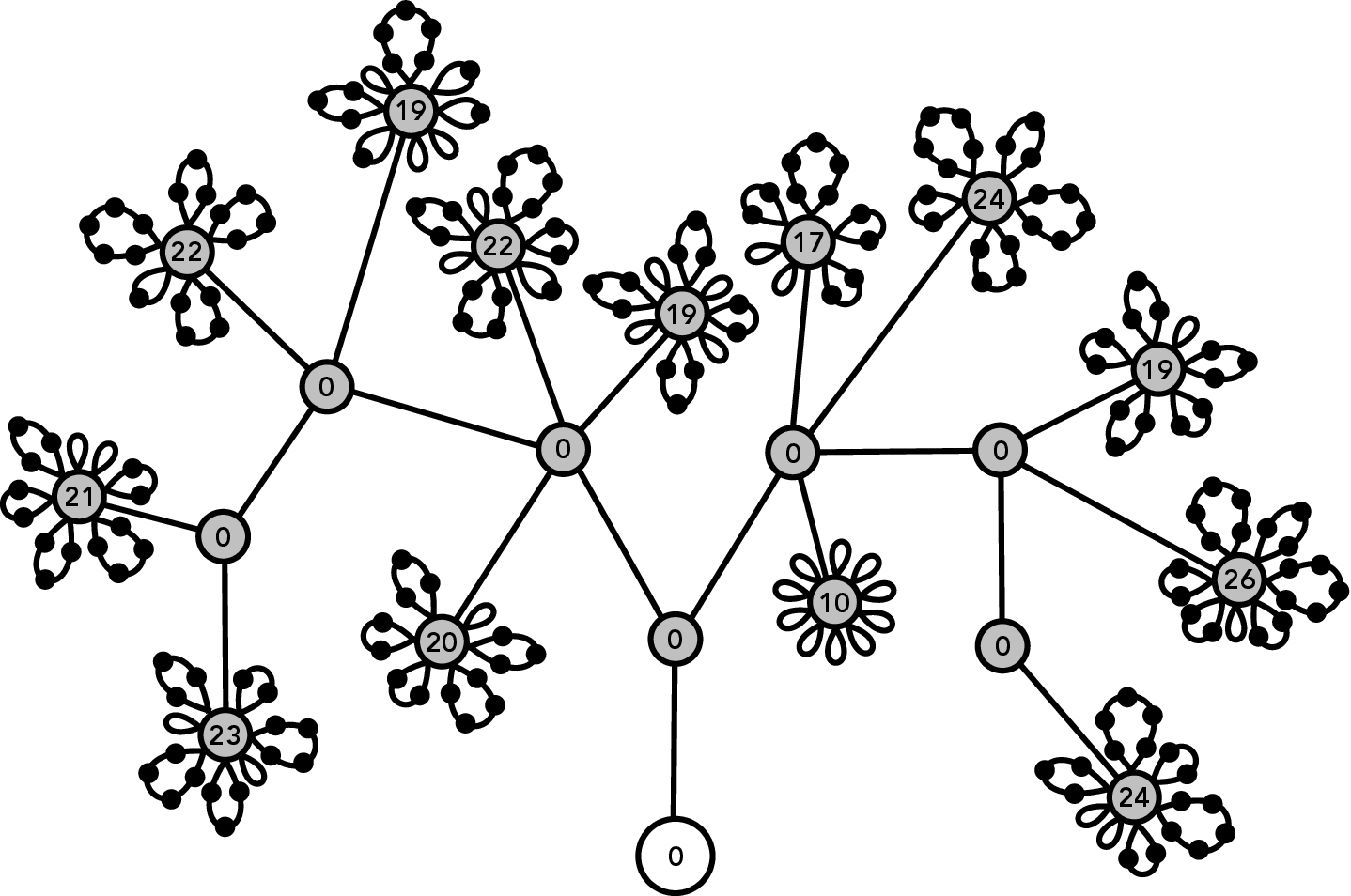}}$ \quad
   $\underset{\textnormal{(b)}}{\includegraphics[width=1.58in]{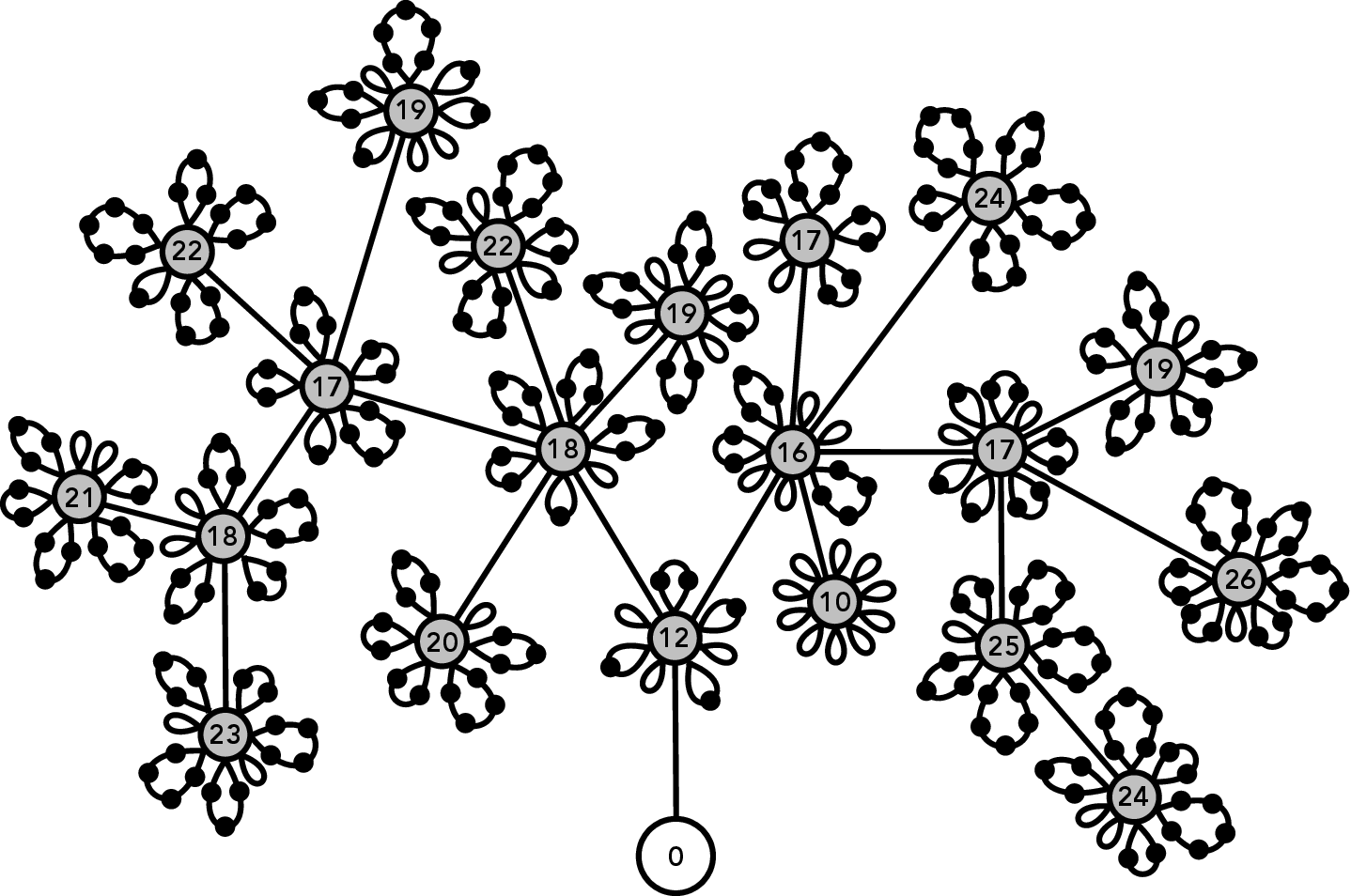}}$ \quad 
   $\underset{\textnormal{(c)}}{\includegraphics[width=1.58in]{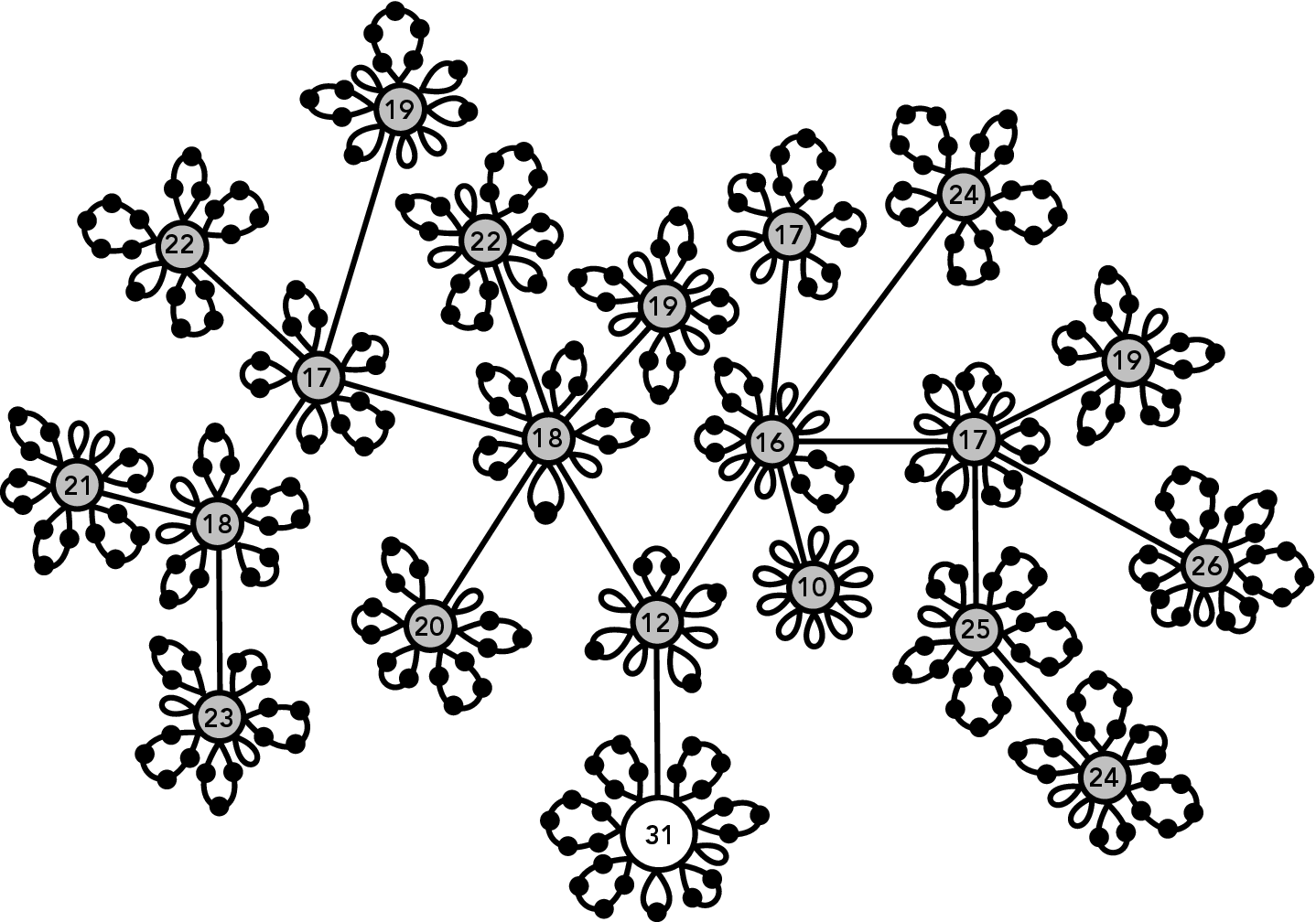}}$ 
   \caption{Three examples of trees with flowers, one for each class $\mathcal R$, $\mathcal S$ and $\mathcal T$ seen in this work, respectively. For all of them the tree's weight is $20$.
   (a) A tree with flowers on the leaves (there are 13 leaves), its size is $409$, the flowers' weight is $389$ and the petals weight is 89. (b) A tree with flowers everywhere but on the root. (c) A tree with flowers everywhere.}
   \label{fig:tree}
\end{figure}

Trees and flowers admit several versions, e.g. plane, non-plane, rooted, etc.
In this note, henceforth we only consider rooted-plane trees
and the flowers will be of two kinds: rooted-plane and non-plane. In fact, rooted-plane and non-plane flowers are combinatorially isomorphic to integer compositions and integer partitions, respectively.
Thus we are simply attaching integer compositions and integer
partitions to the nodes of rooted-plane trees.
We consider restrictions on the number of descendants of the trees,
like binary trees and 2-trees. We also look at 1-trees and arbitrary trees, i.e. paths and no restrictions at all, respectively.
We also consider restrictions on the size of the petals of the flowers, like
binary flowers, $k$-flowers, etc.

Several classes of plane trees are well known to be
isomorphic to many other combinatorial structures,
and the same occurs when we attach flowers to them:
In the On-Line Encyclopedia of Integer Sequences
(OEIS, see \cite{oeis2006}),
we searched for coefficients of generating functions 
and found 
some that are associated to either classes of trees with flowers
or parameters in classes of trees with flowers
(as cumulative generating functions)\footnote{The
tables described in Appendix \ref{sec:appendix} summarize
our findings in the records of OEIS.}.
As a consequence, now we can translate the structures of trees with flowers through
combinatorial isomorphisms and obtain wider combinatorial interpretations
of other combinatorial classes.

The analysis to determine the asymptotic 
growth of the coefficients of the generating functions
associated to plane trees with flowers depends on the particular class 
being considered and hence specific techniques are applied.
In all our cases three different situations can arise:
\begin{enumerate}
	\item
	When the trees are arbitrary trees, 2-trees, or binary trees, the
	generating functions are amenable of singularity analysis, with singularities
	of square root type (see \ref{prop:gftrees:1}, \ref{prop:gftrees:3} and \ref{prop:gftrees:4} in Proposition \ref{prop:gftrees}). Theorem~\ref{theo:analysis_single} summarizes these
	cases and we apply it to produce new examples in sections
	\ref{sec:appendix:1}, \ref{sec:appendix:2} and
	\ref{sec:appendix:3} of Appendix~\ref{sec:appendix}.
	\item
	When the trees are paths (i.e. 1-trees),
	the generating functions are \emph{rational}
	as long as the generating functions of the class
	of allowed flowers on the nodes of
	the $1$-trees are rational (see \ref{prop:gftrees:2} in Proposition \ref{prop:gftrees}). For example,
	this is the case when the set of allowed sizes for the petals is finite,
	or simply when the flowers are rooted-plane with no restrictions.
	Theorems \ref{thm:meromorphicTrees} and
	\ref{thm:meromorphicNonPlane} summarize these cases
	and again in section \ref{sec:appendix:4} of Appendix~\ref{sec:appendix}
	we will present several new examples.
	\item
	When the class consists of $1$-trees (paths) with non-plane flowers
	on the leaves (then there is only one leave actually, and thus only one
	flower, possibly empty) and the set of allowed sizes for the petals is
	infinite, we get certain infinite products that can be analyzed with
	techniques that involve Mellin transformations, residue analysis, and saddle point method,
	e.g. when no restrictions on the non-plane flowers are imposed.
	Here we will not address this situation because all the cases we consider
	are already well known
	(see the Appendix \ref{sec:appendix} and the final section with closing remarks).
\end{enumerate}

Integer partitions, integer compositions, and trees have been widely studied
and they continue to be active research areas, for their own sake and for their applications.
All of them are considered in \cite{FlajoletSedgewick09}, which is a general reference to analytic combinatorics.
A reference that is more focused on (random) trees is \cite{Drmota09}, see also
\cite{Drmota15} and for further developments see \cite{DrmotaJin14,DrmotaJin16, DrmotaJin16_2}. A standard reference to the theory of integer partitions is
Andrews' book \cite{Andrews84}. For more recent developments on integer partitions
see e.g. \cite{BellBurris06, LucaRalaivaosaona16, LipnikMadritsch24}. Integer partitions and trees have been studied together in
\cite{FenerLoizou80,FenerLoizou83, Schmidt2002, ChengLiu2012}.
More recent works on integer compositions include
\cite{
HitczenkoKnopfmacher2005,
BodiniFusyPivoteau10,
BonaKnopfmacher2010,
ArchibaldKnopfmacher2011, 
ArchibaldKnopfmacher2011,
Shapcott2011,
Blecher2012,
LouchardProdinger2013,
Glass2014,
Gafni2015,
MontgomeryTenenbaum2017,
Munagi2018,
Mabry2019,
ArchibaldBlecherKnopfmacherMays2020,
DicskayaMenken23},
and those that study both integer partitions and integer compositions can also be found, e.g.
\cite{Blecher2012, MunagiSellers2018, EngenVatter2019}.
These references are far from exhaustive.

\bigskip

The rest of the paper is organized as follows. In section \S\ref{sec:treesWithFlowers}
we give the combinatorial specifications of trees with flowers, with the corresponding
translation to generating functions. In this section we also present the bivariate
generating functions associated to two parameters on the flowers of trees:
the total number of petals and the total number of edges in all the petals.
In section \S\ref{sec:SingularityAnalysis} we present the asymptotic analysis of the
coefficients of trees with flowers for the first two cases above. In section
\S\ref{sec:finalRemarks} we briefly present final concluding remarks.
Appendix~\ref{sec:appendix} contains tables that summarize which classes
have been registered in OEIS, also those that lack an asymptotic description
despite being registered. In order to exemplify
the results in section \S\ref{sec:SingularityAnalysis},
in the Appendix~\ref{sec:appendix} we also work out all the later cases.

\section{Trees with flowers}
\label{sec:treesWithFlowers}

In this section we define all the combinatorial objects we will analyze.
See \cite{FlajoletSedgewick09} for background.

\subsection{Trees}
\label{sec:treesWithFlowers:trees}

A \emph{tree} will always mean a \emph{rooted-plane tree}.
The size of a tree is its number of edges.
Let $\mathcal K \subseteq \mathbb N^*\triangleq \{1, 2, 3, \ldots \}$
be a subset of the set of positive integers
and then let $\mathcal A$ be the class of all trees such that
the number of descendants of each internal node belongs to $\mathcal K$.
A recursive combinatorial specification of $\mathcal A$ is
\begin{align}
	\label{eq:trees:cs}
	\mathcal{A} \triangleq \mathcal E + \sum\limits_{k\in \mathcal K}(\mathcal Z \times \mathcal{A})^k,
\end{align}
where $\mathcal E$ and $\mathcal Z$ denote a neutral class and an atomic class, respectively.
The generating function of $\mathcal K$ is  $K(z) \triangleq \sum_{k\in\mathcal K}z^k$.
Then the generating function $A(z)$ of $\mathcal A$
satisfies
\begin{align}
	\label{eq:trees:gf}
	A(z) = 1 + K\big(zA(z)\big) .
\end{align}
The functional equation \eqref{eq:trees:gf} on $A(z)$ is linear when
$\mathcal K = \{1\}$ (in this case, the trees are in fact paths),
and is quadratic when $\mathcal K$ equals either
$\mathbb N^*$, $\{2\}$, or $\{1,2\}$, that is, for arbitrary trees, 2-trees,
and binary trees, respectively. In all these cases, simple explicit (and well known) formulas for $A(z)$
are obtained, namely,
\begin{align}
	\mathcal K &= \{1\}  & \Rightarrow \ \ A(z) &= \frac{1}{1-z}  & \oeis{A000012}\\
	\mathcal K &= \{2\}  & \Rightarrow \ \  A(z) & = \frac{1-\sqrt{1-4z^2}}{2z^2}  & \oeis{A000108}^{\dagger} \\
	\mathcal K &= \{1,2\} & \Rightarrow \ \  A(z) & = \frac{1-z-\sqrt{1-2z-3z^2}}{2z^2}  & \oeis{A001006}\\
	\mathcal K &= \mathbb N^* & \Rightarrow \ \ A(z) & = \frac{1-\sqrt{1-4z}}{2z}  & \oeis{A000108}
\end{align}
Above, on the right hand side, the OEIS' records of the corresponding sequences of coefficients of $A(z)$ are shown\footnote{$\dagger$ means that the sequence it is ``essentially'' the same, e.g. in this case the sequence is multiplied by two.}.
These will be the four different classes of trees we will consider.

\subsection{Flowers}
\label{sec:treesWithFlowers:flowers}

Let $\mathcal N \subseteq \mathbb N^*$ be a nonempty subset of the set of
positive integers that
will represent the set of allowed sizes for petals in flowers. Thus an
\emph{$\mathcal N$-flower} will be a flower such that the size of each of its petals
belongs to $\mathcal N$. Let $N(z) \triangleq \sum_{n \in \mathcal N} z^n$ be the generating function of $\mathcal N$ (observe that $N(0)=0$).
We will also restrict $\mathcal N$ to special cases, like
$\{k\}$, $\{1,2\}$, and the whole set $\mathbb N^*$ of positive integers.

\subsubsection{Non-plane flowers: integer partitions}
\label{sec:treesWithFlowers:flowers:non-plane}

\emph{Non-plane flowers}
are combinatorially isomorphic to the class $\mathcal P$
of integer partitions.
Indeed, a non-plane flower is completely determined
by a decreasing sequence of positive integers
$n_1 \geq n_2 \geq \dots \geq n_r$, each representing
the size of each of the petals, and its size
is $n = n_1+n_2+\dots+n_r$, they can be represented e.g. by
Ferrer (also Young) diagrams, or in our case, by non-plane flowers.
Thus, a flower is an integer partition,
and a \emph{petal} is a summand in a partition.
A combinatorial specification and the generating function
of the class $\mathcal P^{\mathcal N}$
of non-plane $\mathcal N$-flowers are
\begin{align}
	\label{eq:partitions}
	\mathcal P^{\mathcal N} \triangleq \textsc{MSet}\left(\sum\limits_{n\in\mathcal N}\mathcal Z^n\right)
	\qquad \textnormal{ and } \qquad
	P^{\mathcal N}(z) \triangleq \prod\limits_{n\in \mathcal N} \frac{1}{1-z^n}.
\end{align}
In particular, let $P(z) \triangleq P^{\mathbb N^*}(z)=\prod_{n\geq 1}1/(1-z^n)$ be the integer partition function.

\subsubsection{Rooted-plane flowers: integer compositions}
\label{sec:treesWithFlowers:flowers:plane}

In a \emph{rooted-plane flower}, the petals are ordered, in particular there is a first petal called the \emph{root}.
A rooted-plane flower is completely determined by a sequence of positive integers $(n_1,\ldots ,n_r)$
and its size is $n=n_1+\ldots + n_r$.
Thus, rooted-plane flowers are combinatorially isomorphic to the class
$\mathcal C$ of integer compositions.
A combinatorial specification and the generating function of
the class $\mathcal C^{\mathcal N}$ of rooted-plane
$\mathcal N$-flowers are
\begin{align}
\label{eq:flowersRP}
	\mathcal C^{\mathcal N} \triangleq \textsc{Seq}\left(\sum\limits_{n\in\mathcal N}\mathcal Z^n\right)
	\qquad \textnormal{ and } \qquad
	C^{\mathcal N}(z) \triangleq \frac{1}{1-N(z)}.
\end{align}

\subsection{Trees with flowers and their recursive specifications}
\label{sec:treesWithFlowers:treesWithFlowers}

Now we put flowers on trees.
Let $\mathcal F$ denote either $\mathcal P^{\mathcal N}$ or $\mathcal C^{\mathcal N}$
and refer to its elements as \emph{$\mathcal F$-flowers}. 
Thus, when we refer to flowers in $\mathcal F$, they can be either non-plane or rooted-plane,
depending on the context that should always be clear.
Let $F(z)$ be the corresponding generating  function of $\mathcal F$.
By the definitions, there is always a neutral class $\mathcal E \in \mathcal F$
that we refer to as the \emph{empty flower} (thus, in particular, $F(0)=1$.)
We let $\mathcal F^*\triangleq \mathcal F \setminus \{\mathcal E\}$ be the class of non-empty flowers,
and thus the generating function of $\mathcal F^*$ is $F^*(z) \triangleq F(z)-1$.
We define the following three main classes of trees with flowers:
\begin{itemize}
	\item
	The class $\mathcal R$ of \emph{$\mathcal K$-trees with $\mathcal F$-flowers on the leaves},
	defined by the recursive combinatorial specification 
	\begin{align}
	\label{eq:treeR}
		\mathcal R \triangleq \mathcal F + \sum\limits_{k \in \mathcal K}(\mathcal Z \times \mathcal R)^k
	\end{align}
	that translated yields the following functional equation on the generating function $R(z)$ of $\mathcal R$:
	\begin{align}
		R(z) = F(z) + K \big( zR(z)\big).
	\end{align}
	\item
	The class $\mathcal S$ of \emph{$\mathcal K$-trees with $\mathcal F$-flowers
	everywhere but on the root}, defined by the recursive combinatorial specification
	\begin{align}
		\mathcal S\triangleq \mathcal E
		+ \sum\limits_{k \in \mathcal K}(\mathcal Z \times \mathcal F \times \mathcal S)^k
	\end{align}
	that translated yields the following functional equation on the generating function $S(z)$ of $\mathcal S$:
	\begin{align}
		S(z) = 1 + K\big( zF(z)S(z)\big).
	\end{align}
	\item
	The class of \emph{$\mathcal K$-trees with $\mathcal F$-flowers everywhere},
	defined by the combinatorial specification
	\begin{align}
		\mathcal T \triangleq \mathcal F \times \mathcal S
	\end{align}
	that translated yields
	\begin{align}
		T(z) & = F(z)S(z).
	\end{align}
\end{itemize}
In addition, we also define the \emph{$*$-classes} $\mathcal R^*$, $\mathcal S^*$ and $\mathcal T^*$
with combinatorial specifications as $\mathcal R$, $\mathcal S$ and $\mathcal T$ above, respectively,
but with the class $\mathcal F$ replaced by $\mathcal F^*$ so that
empty flowers are not allowed (i.e. whenever a flower can be attached to a node in a tree with flowers, then such node has at least one petal attached to it). Thus, the $*$-versions ares classes of \emph{trees with no empty flowers}. The corresponding generating functions
$R^*(z)$, $S^*(z)$ and $T^*(z)$ are also as $R(z)$, $S(z)$ and $T(z)$ above, respectively,
but with $F(z)$ replaced by $F^*(z)$. For example, $\mathcal T^* \triangleq \mathcal F^*\times \mathcal S^*$
with $\mathcal S^* \triangleq \mathcal E + \sum\limits_{k\in\mathcal K}(\mathcal Z \times \mathcal F^* \times \mathcal S^*)^k$, and thus translations yield $T^*(z)=F^*(z)S^*(z)$ with $S^*(z)$ given implicitly by the functional equation $S^*(z) = 1 + \sum\limits_{k\in\mathcal K}\big(zF^*(z)S^*(z)\big)^k$.

Henceforth $f(z)$ will denote $R(z)$, $S(z)$ or $T(z)$, also $R^*(z)$, $S^*(z)$ or $T^*(z)$.

\subsection{Blooming specific types of trees}
\label{subsec:CaseStudies}

Let us present explicit expressions in terms of
$F(z)$ (and $F^*(z)$) of the generating functions
for specific cases of sets of descendants of trees with flowers.

\begin{proposition} [Generating functions of trees with flowers]
\label{prop:gftrees}
The following hold:
\begin{enumerate}
\item
\label{prop:gftrees:1}
\textsc{Arbitrary number of descendants (trees)}.
When $\mathcal K=\mathbb N^*$,
\begin{align}
f(z) = \left\{
\begin{array}{ll}
\frac{1-z+zF(z) - \sqrt{z^2F(z)^2-2(z+z^2)F(z)+(1-z)^2}}{2z}
& \textnormal{ for $\mathcal R$,} \\
\frac{1-\sqrt{1-4zF(z)}}{2zF(z)}
& \textnormal{ for $\mathcal S$ and} \\
\frac{1-\sqrt{1-4zF(z)}}{2z}
& \textnormal{ for $\mathcal T$.} \\
\end{array}
\right.
\end{align}
\item
\label{prop:gftrees:3}
\textsc{Two descendants (2-trees).}
When $\mathcal K=\{2\}$,\\
\begin{align}
f(z) = \left\{
\begin{array}{ll}
\frac{1-\sqrt{1-4z^2F(z)}}{2z^2} & \textnormal{ for $\mathcal R$,} \\
\frac{1-\sqrt{1-4z^2F(z)^2}}{2z^2F(z)^2} & \textnormal{ for $\mathcal S$ and} \\
\frac{1-\sqrt{1-4z^2F(z)^2}}{2z^2F(z)} & \textnormal{ for $\mathcal T$}.
\end{array}
\right.
\end{align}
\item
\label{prop:gftrees:4}
\textsc{At most two descendants (binary trees).}
When $\mathcal K=\{1,2\}$,\\
\begin{align}
f(z)=\left\{
\begin{array}{ll}
\frac{1-z-\sqrt{(1-z)^2-4z^2F(z)}}{2z^2} & \textnormal{ for $\mathcal R$},\\ 
\frac{1-zF(z)-\sqrt{1-2zF(z) - 3z^2F(z)^2}}{2z^2F(z)^2} & \textnormal{ for $\mathcal S$ and}\\ 
\frac{1-zF(z)-\sqrt{1-2zF(z) - 3z^2F(z)^2}}{2z^2F(z)} & \textnormal{ for $\mathcal T$}.
\end{array}
\right.
\end{align}
\item
\label{prop:gftrees:2}
\textsc{One descendant (paths).}
When $\mathcal K=\{1\}$,\\
\begin{align}
\label{eq:gftrees:2}
f(z)=\left\{
\begin{array}{ll}
\frac{F(z)}{1-z} & \textnormal{ for $\mathcal R$},\\ 
\frac{1}{1-zF(z)} & \textnormal{ for $\mathcal S$ and} \\ 
\frac{F(z)}{1-zF(z)} & \textnormal{ for $\mathcal T$}.
\end{array}
\right.
\end{align}
\item
$R^*(z)$, $S^*(z)$ and $T^*(z)$ are as $R(z)$, $S(z)$ and $T(z)$ above, respectively, just replace $F(z)$ with $F^*(z)$.
\end{enumerate}
\end{proposition}

\begin{proof}
Each case follows from solving the corresponding functional equation that
results from the corresponding (recursive) combinatorial specifications given in Section \S \ref{sec:treesWithFlowers:treesWithFlowers}.
\end{proof}

\subsection{Parameters on the flowers of trees}

Given a bivariate generating function $f(z,u)$
associated to a parameter $\kappa$,
the cumulative generating function is
\begin{align}
\Omega(z)\triangleq \left.\frac{\partial}{\partial u}f(z,u)\right|_{u=1}.
\end{align}
Let a \emph{parameter on the flowers of trees}
be a parameter on a class of trees with flowers with the property
that its corresponding bivariate generating function
$f(z,u)$ can be obtained from $f(z)$
by replacing $F(z)$ by a
suitable bivariate generating function $F(z,u)$.
We will consider two types of parameters on the flowers of trees.
Henceforth we define $F_u(z)\triangleq\left.\frac{\partial}{\partial u}F(z,u)\right|_{u=1}$
(and also $F_u^*(z) \triangleq F_u(z)$).

\subsubsection{Number of petals}

Let $\mathcal \chi$ be the parameter in a class of trees with flowers that
returns the \emph{number of petals}. Then $\chi$ is a parameter on the flowers of trees because
the corresponding bivariate generating functions $f(z,u)$ is obtained from
the formulas of their generating functions by replacing the term $F(z)$ by
\begin{align}
\label{eq:bivariate}
	F(z,u) = \prod\limits_{n \in \mathcal N} \frac{1}{1-uz^n}
\qquad
\textnormal{ or by }
\qquad
	F(z,u) = \frac{1}{1-uN(z)},
\end{align}
depending on whether flowers are non-plane ($\mathcal F = \mathcal P^{\mathcal N}$) or rooted-plane ($\mathcal F = \mathcal C^{\mathcal N}$), respectively.
For reference, let us write down $F_u(z)$ with respect to $\chi$, for both non-plane  and rooted-plane flowers:

\begin{proposition}
\label{prop:FuChi}
For the parameter $\chi$, if flowers are non-plane, then
\begin{align}
\label{eq:FuChiNP}
F_u(z) = {P^{\mathcal N}(z)}\sum\limits_{n\in\mathcal N} \frac{z^n}{1-z^n},
\end{align}
and if flowers are rooted-plane, then
\begin{align}
\label{eq:FuChiRP}
F_u(z) =  \frac{N(z)}{\big(1-N(z)\big)^2}.
\end{align}
\end{proposition}

\subsubsection{Number of edges in petals}

Let $\mathcal \xi$ be the parameter in a class of trees with flowers that returns the \emph{number of edges in the petals}. Then $\xi$ is a parameter on the flowers of trees because
the corresponding bivariate generating functions $f(z,u)$ is obtained from the formulas of their generating functions by replacing the term $F(z)$ by $F(z,u)=F(uz)$, that is, by either
\begin{align}
	F(z,u) = \prod\limits_{n \in \mathcal N} \frac{1}{1-u^nz^n}
\qquad
\textnormal{ or }
\qquad
	F(z,u) = \frac{1}{1-N(uz)},
\end{align}
depending on whether flowers are non-plane ($\mathcal F = \mathcal P^{\mathcal N}$) or rooted-plane ($\mathcal F = \mathcal C^{\mathcal N}$),
respectively. 
Again, for reference, let us write down $F_u(z)$ with respect to $\xi$ for both rooted-plane and non-plane flowers:

\begin{proposition}
\label{prop:FuXi}
For the parameter $\xi$, if flowers are non-plane, then
\begin{align}
\label{eq:FuXiNP}
F_u(z) = {P^{\mathcal N}(z)}\sum\limits_{n\in\mathcal N} \frac{nz^n}{1-z^n},
\end{align}
and if flowers are rooted-plane, then
\begin{align}
\label{eq:FuXiRP}
F_u(z) = \frac{zN'(z)}{\big(1-N(z)\big)^2}.
\end{align}
\end{proposition}

\subsubsection{Cumulative generating functions}

Let us write down explicitly the cumulative generating functions $\Omega(z)$ in terms of $F(z,u)$ of a parameter on the flowers of trees.

\begin{proposition}[Cumulative generating functions of parameters on the flowers of trees]
\label{prop:gftreesPetals}
If $\kappa$ is a parameter on the flowers of trees on one of the classes in Proposition \ref{prop:gftrees} (so that the corresponding bivariate generating function $f(z,u)$ 
is obtained by replacing $F(z)$ by $F(z,u)$), then the following hold:
\begin{enumerate}
\item
\label{prop:gftreesPetals:1}
\textsc{Arbitrary number of descendants (trees)}.
When $\mathcal K=\mathbb N^*$,
\begin{align}
\Omega(z)=\left\{
\begin{array}{ll}
\frac{(1+z- zF(z))F_u(z)}{2\sqrt{z^2 F(z)^2-2 (z+z^2) F(z)+(1-z)^2}}+
\frac{F_u(z)}{2} & \textnormal{ for $\mathcal R$,} \\
\frac{\big(1-2 z F(z)\big)F_u(z)}{2 z F(z)^2 \sqrt{1-4 z F(z)}}- \frac{F_u(z)}{2 z F(z)^2} & \textnormal{ for $\mathcal S$ and} \\
\frac{F_u(z)}{\sqrt{1-4 z F(z)}}& \textnormal{ for $\mathcal T$.}
\end{array}
\right.
\end{align}
\item
\label{prop:gftreesPetals:3}
\textsc{Two descendants (2-trees).}
When $\mathcal K=\{2\}$,\\
\begin{align}
\Omega(z)=\left\{
\begin{array}{ll}
\frac{F_u(z)}{\sqrt{1-4 z^2 F(z)}} & \textnormal{ for $\mathcal R$,} \\
\frac{\big(1-2 z^2 F(z)^2\big)F_u(z)}{z^2 F(z)^3 \sqrt{1-4 z^2 F(z)^2}}
-\frac{F_u(z)}{z^2 F(z)^3} & \textnormal{ for $\mathcal S$ and} \\
\frac{F_u(z)}{2 z^2 F(z)^2\sqrt{1-4 z^2 F(z)^2}} - \frac{F_u(z)}{2 z^2 F(z)^2} & \textnormal{ for $\mathcal T$.}
\end{array}
\right.
\end{align}
\item
\label{prop:gftreesPetals:4}
\textsc{At most two descendants (binary trees).}
When $\mathcal K=\{1,2\}$,\\
\begin{align}
\Omega(z)=\left\{
\begin{array}{ll}
\frac{F_u(z)}{\sqrt{(1-z)^2-4 z^2 F(z)}} & \textnormal{ for $\mathcal R$,} \\
\frac{\big(2-3zF(z)-3 z^2 F(z)^2\big)F_u(z)}{2 z^2 F(z)^3 \sqrt{1-2 z F(z)-3 z^2 F(z)^2}}
+\frac{(zF(z)-2)F_u(z)}{2 z^2 F(z)^3 } & \textnormal{ for $\mathcal S$ and} \\
\frac{\big(1-zF(z)\big)F_u(z)}{2 z^2 F(z)^2 \sqrt{1-2 z F(z)-3 z^2 F(z)^2}} - \frac{F_u(z)}{2 z^2 F(z)^2 } & \textnormal{ for $\mathcal T$.} 
\end{array}
\right.
\end{align}
\item
\label{prop:gftreesPetals:2}
\textsc{One descendant (paths).}
When $\mathcal K=\{1\}$,\\
\begin{align}
\label{eq:gftreesPetals:2}
\Omega(z)=\left\{
\begin{array}{ll}
\frac{F_u(z)}{1-z} & \textnormal{ for $\mathcal R$,} \\
\frac{z F_u(z)}{(1-z F(z))^2} & \textnormal{ for $\mathcal S$ and} 
\\
\frac{F_u(z)}{(1-z F(z))^2} & \textnormal{ for $\mathcal T$.}
\end{array}
\right.
\end{align}
\item
\textsc{No empty flowers.}
The corresponding cumulative distribution functions $\Omega(z)$ for the classes $\mathcal R^*$, $\mathcal S^*$ and $\mathcal T^*$ are obtained as in 1-4 above, respectively, just replace $F(z)$ by $F^*(z)$.
\end{enumerate}
\end{proposition}

\section{Singularity analysis}
\label{sec:SingularityAnalysis}

The asymptotic behaviour of the coefficients of the generating functions
of the classes of trees with flowers we have seen can
be deduced with several methods
according to the cases. When $\mathcal K$ equals either $\mathbb N^*$, $\{2\}$ or $\{1,2\}$ (see \ref{prop:gftrees:1}, \ref{prop:gftrees:3} and \ref{prop:gftrees:4} in Proposition \ref{prop:gftrees}), the dominant singularities occur within the set of zeros
of the radicands in the corresponding formulas of the generating functions.
These radicands depend on $\mathcal N$. For example, if
$\mathcal N$ equals $\{k\}$, $\{1,2\}$ and $\mathbb N^*$, then
we get generating functions amenable of singularity analysis,
with dominant singularities which are branch points of square root type.
Theorem \ref{theo:analysis_single} covers these situations.

When $\mathcal K=\{1\}$, the generating functions are rational if $F(z)$ is rational,
for example when $\mathcal N$ is finite, or when $\mathcal N$ equals $\mathbb N^*$
and flowers are rooted-plane. Theorems \ref{thm:meromorphicTrees} and \ref{thm:meromorphicNonPlane} cover these situations. 

When $\mathcal K=\{1\}$, $\mathcal N = \mathbb N^*$ and flowers are non-plane,
then for the class $\mathcal R$ the generating functions
are certain infinite products that can be analyzed with well known technique
like the ones used in Meinardus' theorem. We will not address these cases
and again we refer the reader to the last section with further remarks on this.

\subsection{Branch cuts of square root type}

When $\mathcal K$ equals $\mathbb N^*$, $\{2\}$ or $\{1,2\}$
(cases \ref{prop:gftrees:1}, \ref{prop:gftrees:3} and \ref{prop:gftrees:4} in Proposition \ref{prop:gftrees}),
the positive real dominant singularitiy $\zeta$ occurs as a zero of the
radicand $p(z)$ in the formulas for $f(z)$ and $\Omega(z)$ (for example, 
if $\mathcal K = \mathbb N^*$, then $p(z) = z^2F(z)^2-2(z+z^2)F(z) + (1-z)^2$
for $\mathcal R$, $p(z) = 1-4zF(z)$ for both $\mathcal S$ and $\mathcal T$, and $p(z) = 1-4zF^*(z)$ for both $\mathcal S^*$ and $\mathcal T^*$, etc.).
Thus, for all these cases, the asymptotic behavior of the coefficients of the generating functions
can be obtained by the process of singularity analysis
if the generating function is
\emph{amenable} to this process, which means that it
must satisfy the conditions of singularity analysis as expressed in
Theorem VI.4 (single dominant singularity\footnote{In \cite{FlajoletSedgewick09} there is also Theorem VI.5
for multiple dominant singularities. We will encounter multiple dominant singularities in some examples in the Appendix (see cases \oeis{A052702}, \oeis{A023431}). For simplicity we only state Theorem \ref{theo:analysis_single} which is for a single dominant singularity and leave the corresponding statement of Theorem VI.5 as an exercise.}) in \cite{FlajoletSedgewick09}.
Let us state this result in our context.
Recall that a \emph{$\Delta$-domain}
$\Delta_0\subset \mathbb C$ is a domain of the form
\begin{align}
	\Delta_0 = \Delta(\phi,R) \triangleq
	\{ z \ : \ |z|<R, \ z \neq 1, \ |\textnormal{arg}(z-1)| > \phi \}
\end{align}
for some $R>1$ and $0 < \phi < \frac{\pi}{2}$, and that
$\zeta \cdot \Delta_0$ denotes the image
by the mapping $z\mapsto \zeta z$, with $\zeta \in \mathbb C$.
	
\begin{theorem} [Single singularity analysis of trees with flowers]
	\label{theo:analysis_single}
	Suppose that $\mathcal K$ equals either $\mathbb N^*$, $\{2\}$ or $\{1,2\}$.
	Let $\zeta$ be the real dominant singularity of $f(z)$, it is 
	the smallest positive root of $p(z)$.
	Suppose that $f(z)$ can be continued to a domain of the form
	$\zeta \cdot \Delta_0$, where $\Delta_0$ is a $\Delta$-domain.
	In addition, suppose that $\zeta$ is a simple zero of $p(z)$ and let
	\begin{align}
		\theta(n) = \frac{\sqrt{-\zeta p'(\zeta)}}{2\sqrt{\pi n^3}}
		\qquad
		\textnormal{ and }
		\qquad
		\vartheta(n) = \frac{1}{\sqrt{-\zeta p'(\zeta)\pi n}}.
	\end{align}
	\begin{enumerate}
	\item \textsc{Arbitrary number of descendants (trees)}.
		If $\mathcal K = \mathbb N^*$, then
		\begin{align}
			[z^n]f(z) \sim
			\left\{
			\begin{array}{rl}
				\frac{1}{2\zeta} \cdot \theta(n) \cdot \zeta^{-n} & \textnormal{ \emph{ for $\mathcal R$,}} \\
				\frac{1}{2\zeta F(\zeta)} \cdot \theta(n)\cdot \zeta^{-n} & \textnormal{ \emph{ for $\mathcal S$ and}}  \\ 
				\frac{1}{2\zeta} \cdot \theta(n) \cdot \zeta^{-n} & \textnormal{ \emph{ for $\mathcal T$.}} \\
			\end{array}
			\right.
		\end{align}
		Also, for a parameter $\kappa$ on the flowers of trees,
		\begin{align}
			[z^n]\Omega(z) \sim
			\left\{
			\begin{array}{rl}
				\frac{\big(1+\zeta-\zeta F(\zeta)\big)F_u(\zeta)}{2} \cdot \vartheta(n) \cdot \zeta^{-n} & \textnormal{ \emph{ for $\mathcal R$,}} \ \\ 
				\frac{\big(1-2\zeta F(\zeta)\big)F_u(\zeta)}{2\zeta F(\zeta)^2} \cdot \vartheta(n) \cdot \zeta^{-n} & \textnormal{ \emph{ for $\mathcal S$ and}} \\
				F_u(\zeta) \cdot \vartheta(n) \cdot \zeta^{-n} & \textnormal{ \emph{ for $\mathcal T$,}} \\
			\end{array}
			\right.
		\end{align}
		hence
		\begin{align}
			\mathbb E_{\Theta_n} (\kappa) 
			\sim
			\left\{
			\begin{array}{rl}
				\frac{2\big(1+\zeta-\zeta F(\zeta)\big)F_u(\zeta)}{-p'(\zeta)} \cdot n
				& \textnormal{ \emph{ if $\Theta = \mathcal R$,}} \ \\
				\frac{2\big(1-2\zeta F(\zeta)\big)F_u(\zeta)}{-\zeta F(\zeta)p'(\zeta)} \cdot  n
				& \textnormal{ \emph{ if $\Theta = \mathcal S$ and}} \\
				\frac{4F_u(\zeta)}{-p'(\zeta)} \cdot n & \textnormal{ \emph{ if $\Theta = \mathcal T$.}} \\
			\end{array}
			\right.
		\end{align}
	\item \textsc{Two descendants (2-trees).}
		If $\mathcal K = \{2\}$, then
		\begin{align}
			[z^n]f(z) \sim
			\left\{
			\begin{array}{rl}
				\frac{1}{2\zeta^2} \cdot \theta(n) \cdot \zeta^{-n} & \textnormal{ \emph{ for $\mathcal R$,}} \ \\ 
				\frac{1}{2\zeta^2 F(\zeta)^2} \cdot \theta(n) \cdot \zeta^{-n} & \textnormal{ \emph{ for $\mathcal S$ and}} \ \\ 
				\frac{1}{2\zeta^2 F(\zeta)} \cdot \theta(n)\cdot \zeta^{-n} & \textnormal{ \emph{ for $\mathcal T$.}}
			\end{array}
			\right.
		\end{align}
		Also, for a parameter $\kappa$ on the flowers of trees,
		\begin{align}
			[z^n]\Omega(z) \sim
			\left\{
			\begin{array}{rl}
				F_u(\zeta) \cdot \vartheta(n) \cdot \zeta^{-n} & \textnormal{ \emph{ for $\mathcal R$,}} \ \\ 
				\frac{\big(1-2\zeta^2 F(\zeta)^2\big)F_u(\zeta)}{\zeta^2 F(\zeta)^3} \cdot \vartheta(n) \cdot \zeta^{-n} & \textnormal{ \emph{ for $\mathcal S$ and}} \\
				\frac{F_u(\zeta)}{2\zeta^2F(\zeta)^2} \cdot \vartheta(n) \cdot \zeta^{-n} & \textnormal{ \emph{ for $\mathcal T$,}} \\
			\end{array}
			\right.
		\end{align}
		 hence
		\begin{align}
			\mathbb E_{\Theta_n} (\kappa) 
			\sim
			\left\{
			\begin{array}{rl}
				\frac{4\zeta F_u(\zeta)}{-p'(\zeta)} \cdot n
				& \textnormal{ \emph{ if $\Theta = \mathcal R$,}} \ \\ 
				\frac{2\big(1-2\zeta F(\zeta)\big)F_u(\zeta)}{-\zeta F(\zeta)p'(\zeta)} \cdot  n 
				& \textnormal{ \emph{ if $\Theta = \mathcal S$ and}} \\
				\frac{4 F_u(\zeta)}{-p'(\zeta)} \cdot n & \textnormal{ \emph{ if $\Theta = \mathcal T$.}} \\
			\end{array}
			\right.
		\end{align}
	\item \textsc{At most two descendants (binary trees).}
		If $\mathcal K = \{1,2\}$, then
		\begin{align}
			[z^n]f(z) \sim
			\left\{
			\begin{array}{rl}
				\frac{1}{2\zeta^2} \cdot \theta(n) \cdot \zeta^{-n} & \textnormal{ \emph{ for $\mathcal R$,}} \ \\
				\frac{1}{2\zeta^2 F(\zeta)^2} \cdot \theta(n) \cdot \zeta^{-n} & \textnormal{ \emph{ for $\mathcal S$ and}} \ \\ 
				\frac{1}{2\zeta^2 F(\zeta)} \cdot \theta(n) \cdot \zeta^{-n} & \textnormal{ \emph{ for $\mathcal T$.}}
			\end{array}
			\right.
		\end{align}
		Also, for a parameter $\kappa$ on the flowers of trees,
		\begin{align}
			[z^n]\Omega(z) \sim
			\left\{
			\begin{array}{rl}
				F_u(\zeta) \cdot \vartheta(n) \cdot \zeta^{-n} & \textnormal{ \emph{ for $\mathcal R$,}} \ \\
				\frac{\big(2-3\zeta F(\zeta) -3\zeta^2F(\zeta)^2\big)F_u(\zeta)}{2\zeta^2 F(\zeta)^3} \cdot \vartheta(n) \cdot \zeta^{-n} & \textnormal{ \emph{ for $\mathcal S$ and}} \\
				\frac{\big(1-\zeta F(\zeta)\big)F_u(\zeta)}{2\zeta^2F(\zeta)^2} \cdot \vartheta(n) \cdot \zeta^{-n} & \textnormal{ \emph{ for $\mathcal T$,}} \\
			\end{array}
			\right.
		\end{align}
		hence
		\begin{align}
			\mathbb E_{\Theta_n} (\kappa) 
			\sim
			\left\{
			\begin{array}{rl}
				\frac{4\zeta^2 F_u(\zeta)}{-\zeta p'(\zeta)} \cdot n
				& \textnormal{ \emph{ if $\Theta = \mathcal R$,}} \ \\ 
				\frac{2\big(2-3\zeta F(\zeta) -3\zeta^2F(\zeta)^2\big)F_u(\zeta)}{-\zeta F(\zeta)p'(\zeta)} \cdot n 
				& \textnormal{ \emph{ if $\Theta = \mathcal S$ and}} \\
				\frac{2\big(1-\zeta F(\zeta)\big)F_u(\zeta)}{-\zeta F(\zeta)p'(\zeta)} \cdot n & \textnormal{ \emph{ if $\Theta = \mathcal T$.}} \\
			\end{array}
			\right.
		\end{align}
		\item \textsc{No empty flowers.}
		The corresponding asymptotic growths
		for the classes $\mathcal R^*$, $\mathcal S^*$ and $\mathcal T^*$ and
		a parameter $\kappa$ on the flowers of trees are as in 1-3 above,
		respectively, just replace $F(z)$ by $F^*(z)$.
	\end{enumerate}
\end{theorem}

\begin{proof}
	By l'H\^opital's rule,
	$$
		\lim\limits_{z\to \zeta } \frac{p(z)}{\zeta-z} = -p'(\zeta)
	$$
	and hence
	the result follows from Theorem~VI.1 in \cite{FlajoletSedgewick09}.
\end{proof}

\subsection{Meromorphic extensions}

When $\mathcal K = \{1\}$, $F(z)$ can be a rational generating function, for example
when $\mathcal N$ is finite, or when $\mathcal N = \mathbb N^*$ and the flowers are rooted-plane.
Other restriction sets $\mathcal N$ may yield generating functions that admit
meromorphic extensions on discs of radius greater than their radii
of convergence. When this occurs,
we can apply Theorem IV.9 in \cite{FlajoletSedgewick09}. Let us adapt this theorem to our particular context.

\begin{theorem} [Asymptotics for meromorphic functions of $1$-trees with flowers]
\label{thm:meromorphicTrees}
Assume that $\mathcal K = \{1\}$.
Suppose that $f(z)$ has a single dominant singularity at $\alpha>0$ which is a pole of order $r\geq 1$.
Then the following hold: 
\begin{enumerate}
	\item
	\label{eq:Mero1}
		Consider the class $\mathcal R$. If $\mathcal N=\{k\}$,
		then $\alpha=1$ and $r=2$. Hence
		\begin{align}
		\label{eq:Rk}
			[z^n]R(z)  \sim \frac{n}{k}
		\end{align}
		and also
		\begin{align}
		\label{eq:RkPara}
			[z^n]\Omega(z)  \sim \left\{
			\begin{array}{ll}
			\frac{n^2}{2k^2} & \textnormal{ for $\chi$ and}\\
			\frac{n^2}{2k} & \textnormal{ for $\xi$},
			\end{array}
			\right.
		\end{align}
		thus
		\begin{align}
		\mathbb E_{\mathcal R_n}(\chi) = \frac{n}{2k}
		\quad \textnormal{ and } \quad
		\mathbb E_{\mathcal R_n}(\xi) = \frac{n}{2}.
		\end{align}
		Otherwise, if $\mathcal N$ possesses at
		least two elements and flowers are rooted-plane, then $0<\alpha < 1$ is the root of $1-N(z)$, it
		is a simple pole (i.e. $r=1$),
		\begin{align}
		\label{eq:R2f}
			[z^n]R(z)  \sim\frac{1}{\alpha(1-\alpha)N'(\alpha)}\alpha^{-n} 
		\end{align}
		and also
		\begin{align}
		\label{eq:R2Para}
			[z^n]\Omega(z) 
			\sim\left\{
			\begin{array}{ll}
			 \frac{N(\alpha)}{\alpha^{2}(1-\alpha)N'(\alpha)^2}n\alpha^{-n}
			 & \textnormal{ for $\chi$ and} \\
			 \frac{1}{\alpha (1-\alpha)N'(\alpha)}n\alpha^{-n}
			 & \textnormal{ for $\xi$,}
			\end{array}
			\right.
		\end{align}
		hence
		\begin{align}
		\mathbb E_{\mathcal R_n}(\chi) = \frac{N(\alpha)}{\alpha N'(\alpha)}n
		\quad \textnormal{ and } \quad
		\mathbb E_{\mathcal R_n}(\xi) = \frac{1}{N'(\alpha)}n.
		\end{align}
		\item
		\label{eq:Mero2}
		Consider the class $\mathcal S$. Then we have that $0 < \alpha < 1$ is the root of $1-zF(z)$,
		again it is a simple pole,
		\begin{align}
		\label{eq:S2}
			[z^n]S(z) & \sim \frac{1}{\alpha \big(F(\alpha)+\alpha F'(\alpha)\big)} \alpha^{-n}
		\end{align}
		and also, for a parameter $\kappa$ on the flowers of trees, we always have
		\begin{align}
		\label{eq:SP2}
			[z^n]\Omega(z)
			\sim \frac{F_u(\alpha)}{\alpha\big(F(\alpha)+\alpha F'(\alpha) \big)^2}n \alpha^{-n} ,
		\end{align}
		hence
		\begin{align}
			\mathbb E_{\mathcal S_n}(\kappa) \sim \frac{F_u(\alpha)}{\big(F(\alpha)+\alpha F'(\alpha) \big)}n  .
		\end{align}
		\item
		\label{eq:Mero3}
		For the class $\mathcal T$, again we have that $0 < \alpha < 1$ is the positive root of $1-zF(z)$, it is a simple pole,
		\begin{align}
			[z^n]T(z) & \sim \frac{F(\alpha)}{\alpha \big(F(\alpha)+\alpha F'(\alpha)\big)} \alpha^{-n}
		\end{align}
		and also, for a parameter $\kappa$ on the flowers of trees, we always have
		\begin{align}
			[z^n]\Omega(z)
			\sim \frac{F_u(\alpha)}{\alpha^2\big(F(\alpha)+\alpha F'(\alpha) \big)^2}n \alpha^{-n} ,
		\end{align}
		hence
		\begin{align}
			\mathbb E_{\mathcal T_n}(\kappa) \sim \frac{F_u(\alpha)}{\alpha F(\alpha) \big(F(\alpha)+\alpha F'(\alpha) \big)}n  .
		\end{align}
		\item
		The corresponding asymptotic growths for both
		the classe $\mathcal R^*$, in 1 replace $N(z)$
		by $N^*(z)\triangleq N(z) - 1$, and more generally,
		for the classes $\mathcal S^*$ and $\mathcal T^*$ and
		a parameter on the flowers of trees $\kappa$,
		in 2 and 3,	 replace $F(z)$ by $F^*(z)$\footnote{When working with parameters in the *-classes, the corresponding bivariate generating function is $F^*(z,u)\triangleq F(z,u)-1$ and so we have $F_u^*(z)=F_u(z)$.}.
	\end{enumerate}
\end{theorem}

\begin{proof}
First we prove \ref{eq:Mero1} which is for the class $\mathcal R$. Suppose that $\mathcal N=\{k\}$. Then
\begin{align}
f(z) = \frac{1}{(1-z)(1-z^k)}= \frac{1}{(1-z)^2}\frac{1}{(1+\ldots+z^{k-1})}.
\end{align}
Thus $r=2$ and equation \eqref{eq:Rk} follows from Theorem IV.9 in \cite{FlajoletSedgewick09}
and the fact that
\begin{align}
\lim_{z\to 1}(1+\ldots+z^{k-1})=k.
\end{align}
Similarly,
\begin{align}
\Omega(z) = \left\{
\begin{array}{ll}
\frac{1}{(1-z)^3}\frac{z^{k}}{(1+\ldots+z^{k-1})^2}
& \textnormal{ for $\chi$,}\\
\frac{1}{(1-z)^3}\frac{kz^{k}}{(1+\ldots+z^{k-1})^2}
& \textnormal{ for $\xi$,}
\end{array}
\right.
\end{align}
and hence equation \eqref{eq:RkPara} follows from
Theorem IV.9 in \cite{FlajoletSedgewick09} and the fact that
\begin{align}
\lim_{z\to 1}(1+\ldots+z^{k-1})^2 = k^2.
\end{align}
Now suppose that $\mathcal N$ has at least two elements and flowers are rooted-plane. Then
$0 < \alpha < 1$ because it is the positive root of $1-N(z)$ (recall that  $f(z)=\frac{1}{(1-z)(1-N(z))}$, see equations \eqref{eq:flowersRP} and \eqref{eq:gftrees:2}, and now $\mathcal N$ has at least two elements!).
Furthermore, $\alpha$ is a simple pole of $f(z)$. Indeed,
\begin{align}
\lim\limits_{z\to \alpha}(z-\alpha)f(z) =\lim\limits_{z\to \alpha}\frac{(z-\alpha)}{(1-z)(1-N(z))}
\overset{\textnormal{l'H\^op.}}=\frac{-1}{(1-\alpha)N'(\alpha)} < 0
\end{align} 
because $N'(\alpha) > 0$ (recall that $N(z)$ is the generating function of a non-empty class
and $0<\alpha<1$), thus $r=1$.
Hence equation \eqref{eq:R2f} follows from
Theorem IV.9 in \cite{FlajoletSedgewick09}.
Furthermore, we observe that in this case $\Omega(z)$ also has
a pole at $z=\alpha$ and that its order is $2$ (see equations \eqref{eq:FuChiRP}
and \eqref{eq:FuXiRP} in Propositions \ref{prop:FuChi} and \ref{prop:FuXi},
resp., together with equation \eqref{eq:gftreesPetals:2} in Proposition \ref{prop:gftreesPetals}),
and a similar argument applies to deduce equation \eqref{eq:R2Para}.

To prove \ref{eq:Mero2}, now we have that $\alpha$ is the positive root of
$1-zF(z)$ (see equation \eqref{eq:gftrees:2} in Proposition \ref{prop:gftrees}).
Thus $0 < \alpha < 1$ because $\mathcal N \neq \varnothing$. Furthermore,
$\alpha$ is simple because
\begin{align}
\lim_{z\to \alpha}(z-\alpha)f(z) = \lim_{z\to \alpha}\frac{z-\alpha}{1-zF(z)}
\overset{\textnormal{l'H\^op.}}= \frac{-1}{F(\alpha)+\alpha F'(\alpha)}
\end{align}
and since $F(\alpha)+\alpha F'(\alpha)>0$, $r=1$.
Equation \eqref{eq:S2} now follows from Theorem IV.9 in \cite{FlajoletSedgewick09}.
For equation \eqref{eq:SP2} we argue in a similar way. First we write
\begin{align}
\Omega(z) =  \frac{1}{(z-\alpha)^2}zF_u(z)\left(\frac{z-\alpha}{1-zF(z)}\right)^2
\end{align}
and again by hypothesis, there exists the following non-zero limit:
\begin{align}
\lim_{z\to\alpha} zF_u(z)\left(\frac{z-\alpha}{1-zF(z)}\right)^2 
\overset{\textnormal{l'H\^op.}}=
 \frac{\alpha F_u(\alpha)}{\big( F(\alpha)+\alpha F'(\alpha) \big)^2} .
\end{align}
Similarly, equation \eqref{eq:SP2} again now follows from Theorem IV.9 in \cite{FlajoletSedgewick09}.

To prove \ref{eq:Mero3} and the *-versions, proceed similarly.
\end{proof}

\begin{theorem}[Asymptotics of $1$-trees with non-plane flowers on the leaves and bounded petal size]
\label{thm:meromorphicNonPlane}
	Assume both that $\mathcal K=\{1\}$ and that flowers are non-plane. Then,
	if $\mathcal N$ is finite, then
		\begin{align}
			[z^n]R(z) &\sim [z^n]R^*(z)\sim \left(\prod\limits_{m\in\mathcal N}\frac{1}{m}\right) \cdot \frac{n^{|\mathcal N|}}{|\mathcal N|!}.
		\end{align}
	Also,
	\begin{align}
		[z^n]\Omega (z) \sim \left\{
		\begin{array}{rl}
			\frac{1}{(|\mathcal N|+1)!}\left( \prod\limits_{m \in\mathcal N}\frac{1}{m} \right) \left( \sum\limits_{m \in\mathcal N} \frac{1}{m}\right)n^{|\mathcal N|+1} & \textnormal{ for } \chi, \\
			\frac{|\mathcal N|}{(|\mathcal N|+1)!}\left( \prod\limits_{m \in\mathcal N}\frac{1}{m} \right) n^{|\mathcal N|+1} & \textnormal{ for } \xi,
		\end{array}
		\right.
	\end{align}
	hence
	\begin{align}
		\mathbb E_{\mathcal R_n}(\kappa) \sim \left\{
		\begin{array}{rl}
			\frac{1}{(|\mathcal N|+1)} \left( \sum\limits_{m \in\mathcal N} \frac{1}{m}\right)n & \textnormal{ if } \kappa = \chi, \\
			\frac{|\mathcal N|}{(|\mathcal N|+1)} n & \textnormal{ if } \kappa = \xi.
		\end{array}
		\right.
	\end{align}	
\end{theorem}

\begin{proof}
We write
\begin{align}
R(z) &= \frac{F(z)}{1-z}=\frac{1}{1-z}\prod_{m\in\mathcal N}\frac{1}{1-z^m}\\ &=
\frac{1}{(1-z)^{|\mathcal N|+1}}\prod_{m\in\mathcal N}\frac{1}{1+\ldots + z^{m-1}}
\end{align}
and then use Theorem IV.9  in \cite{FlajoletSedgewick09} together with the fact that
\begin{align}
\label{eq:prodLimit}
\lim\limits_{z\to 1}\prod\limits_{m\in\mathcal N} \frac{1}{1+\ldots + z^{m-1}} = \prod\limits_{m\in\mathcal N}\frac{1}{m}.
\end{align}
For $\mathcal R^*$ the situation is very similar, first we write
\begin{align}
R^*(z) & = \frac{F(z)-1}{1-z}=\frac{1}{1-z}\left(\prod_{m\in\mathcal N}\frac{1}{1-z^m} -1\right) \\
&= \frac{1}{(1-z)^{|\mathcal N|+1}} \cdot \frac{1-\prod_{m\in\mathcal N}(1-z^m)}{\prod_{m\in\mathcal N}(1+\ldots + z^{m-1})}
\end{align}
and then use Theorem IV.9  in \cite{FlajoletSedgewick09} together with the fact that
\begin{align}
\lim\limits_{z\to 1}\frac{1-\prod_{m\in\mathcal N}(1-z^m)}{\prod_{m\in\mathcal N}(1+\ldots + z^{m-1})} = \prod\limits_{m\in\mathcal N}\frac{1}{m}.
\end{align}

Next, we have from equations \eqref{eq:FuChiNP} and \eqref{eq:FuXiNP} that
\begin{align}
	\Omega(z) = \frac{F_u(z)}{1-z} & =
	\left\{
	\begin{array}{ll}
		\frac{1}{1-z}{P^{\mathcal N}(z)}\sum\limits_{n\in\mathcal N} \frac{z^n}{1-z^n}& \textnormal{ for } \chi \\
		\frac{1}{1-z}{P^{\mathcal N}(z)}\sum\limits_{n\in\mathcal N} \frac{nz^n}{1-z^n}& \textnormal{ for } \xi
	\end{array}
	\right. \\
	& = \left\{
	\begin{array}{ll}
		\frac{1}{(1-z)^{|\mathcal N|+2}}\prod\limits_{n\in \mathcal N}\frac{1}{1+\ldots + z^{n-1}}\sum\limits_{n \in \mathcal N}\frac{z^n}{1+\ldots +z^{n-1}} & \textnormal{ for } \chi \\
		  \frac{1}{(1-z)^{|\mathcal N|+2}}\prod\limits_{n\in \mathcal N}\frac{1}{1+\ldots + z^{n-1}}\sum\limits_{n \in \mathcal N}\frac{nz^n}{1+\ldots +z^{n-1}} & \textnormal{ for } \xi.
	\end{array}
	\right.
\end{align}
We already know equation \eqref{eq:prodLimit} and we also have
\begin{align}
	\lim_{z\to 1} \sum\limits_{n\in \mathcal N} \frac{z^n}{1+\ldots + z^{n-1}} = \sum\limits_{n\in \mathcal N} \frac{1}{n}
\end{align}
and
\begin{align}
	\lim_{z\to 1} \sum\limits_{n\in \mathcal N} \frac{nz^n}{1+\ldots + z^{n-1}} = |\mathcal N|.
\end{align}
\end{proof}
\section{Final remarks}
\label{sec:finalRemarks}

When $\mathcal K = \{1\}$, flowers are non-plane, $\mathcal N$ is infinite (e.g. when $\mathcal N =\mathbb N^*$), for the class $\mathcal R$ of $1$-trees
with flowers on the leaves (see \ref{prop:gftrees:2} in Proposition \ref{prop:gftrees}), its $*$-version and their parameters on the flowers of the $1$-trees (see \ref{prop:gftreesPetals:2} in Proposition \ref{prop:gftreesPetals}), the asymptotic analysis of the growth of the coefficients can be also carried out through standard techniques that involve Mellin transformations, residue analysis, and saddle point method. The cases in Table \eqref{eq:Tables:1} that correspond to these situations are all already registered in OEIS together with their asymptotic growth (they all are colored gray $\begin{array}{l}\cellcolor{gray!25} \ \ \end{array}$).
Thus we will not address examples of this, instead we refer the reader
to \cite{AhmadiGomezWard202X}, not only
as another global approach to classify families of combinatorial classes with an asymptotic analysis focused on these techniques, but also to point
out that all what we have seen here can be extrapolated much further,
e.g. to trees with labelled colorful flowers!

\begin{appendices}



\section{}\label{sec:appendix}%

We searched in OEIS the coefficients of $f(z)$ and $\Omega(z)$ for
the three classes $\mathcal R$, $\mathcal S$, $\mathcal T$
and their $*$-versions, for both parameters $\chi$ and $\xi$,
for both non-plane and rooted-plane flowers on $\mathcal K$-trees,
with $\mathcal K$ along the four cases we have considered
($\mathbb N^*$, $\{2\}$, $\{1,2\}$ and $\{1\}$),
for four specific cases of $\mathcal N$
($\mathbb N^*$, $\{1\}$, $\{2\}$ and $\{1,2\}$).
In this Appendix we present our findings grouped in four sections,
one for each case of $\mathcal K$. In each case, we present
a table that summarizes what we found by the time or writing:
\begin{itemize}
\item
We show the OEIS' identifier code
whenever the sequence matches one of our classes of trees with flowers,
and it is colored gray
$\begin{array}{l}\cellcolor{gray!25} \ \ \end{array}$ 
if one can find the asymptotic growth of the coefficients in the description of the sequence in OEIS.
\item
$\color{blue}\circ\color{black}$ means
that the asymptotic equivalence has no explicit reference registered
in OEIS and these cases are colored green
$\begin{array}{l}\cellcolor{green!25} \ \ \end{array}$, and in order to
present some examples of the previous theorems, we will work out all these cases.
\item
$\substack{..\\ \smile}$ means that the sequence would be new to
OEIS because we could not find a registry that matched, and
these cases are colored light green
$\begin{array}{l}\cellcolor{green!5} \ \ \end{array}$.
\item
$\Leftarrow$ means that the sequence is equal to the sequence
that is to its left.
\item
$\dagger$ means as before, i.e. that the sequence is ``essentially'' the same
(e.g. except for the first few terms, by a shift, by an integer multiple, etc.).
\item
$\conj$ means that we found an open conjecture
in OEIS' description of the sequence.
\end{itemize}

\subsection{Case $\mathcal K = \mathbb N^*$.}
\label{sec:appendix:1}

Table A.1 in equation \eqref{eq:Tables:N} summarizes our findings for the case $\mathcal K = \mathbb N^*$. There were three cases among those that we found in OEIS that have no asymptotic description. Let us work these out:

\begin{example} [$\mathcal K = \mathbb N^*$: \oeis{A254314}$^{\dagger, \conj}$, \oeis{A025266}$^\dagger$, \oeis{A026571}] Suppose that $\mathcal K= \mathbb N^*$.
If $\mathcal N = \mathbb N^*$ and flowers are rooted-plane, then $F(z) = (1-z)/(1-2z)$ and in this situation we have the following two cases:
\begin{itemize}
\item \oeis{A254314}$^{\dagger, \conj}$. 
For the class $\mathcal R$ we have
\begin{align}
\label{eq:p:A254314}
p(z) = \frac{(z-1)(z^3-11z^2+7z-1)}{(1-2z)^2}.
\end{align}
This rational function has a pole of order $2$ at $z=1/2=0.5$ and also four distinct positive roots that can be computed explicitly, in particular the dominant singularity of $f(z)$ is
\begin{align}
\label{eq:zetaA254314}
\zeta & =
\frac{1}{3} \left(-\sqrt[3]{-998+6 i \sqrt{111}}-\frac{100}{\sqrt[3]{-998+6 i \sqrt{111}}}+11\right)
\approx 0.2123\ldots.
\end{align}
\footnotesize
\begin{align}
	\label{eq:Tables:N}
	\begin{array}{|c|}
	\hline \\
	\begin{array}{|| c | c ||| c | c || c | c || c | c ||}
		\hline
		\textnormal{Class} & \mathcal N & \textnormal{rooted-plane} & \substack{\textnormal{non-}\\ \textnormal{plane}}
			& \substack{\#\textnormal{ petals}\\ \textnormal{rooted-plane}} 
			& \substack{\#\textnormal{ petals}\\ \textnormal{non-plane}}
			& \substack{\#\textnormal{ edges}\\ \textnormal{in petals}\\ \textnormal{rooted-plane}} 
			& \substack{\#\textnormal{ edges} \\ \textnormal{in petals}\\ \textnormal{non-plane}} \\
		\hline \hline \hline
		\mathcal R & \mathbb N^* & \cellcolor{green!25} \oeis{A254314}^{\dagger,\conj,\noa} & \na & \na & \na & \na & \na \\
		\hline
		\mathcal S & \mathbb N^* & \cellcolor{gray!25} \oeis{A059278} & \na & \na & \na & \na & \na \\
		\hline
		\mathcal T & \mathbb N^* & \cellcolor{gray!25} \oeis{A059279}^{\conj} & \na & \na & \na & \na & \na \\
		\hline
		\mathcal R^* & \mathbb N^* & \na & \na & \na & \na & \na & \na \\
		\hline
		\mathcal S^* & \mathbb N^* & \cellcolor{green!25} \oeis{A025266}^{\dagger,\noa} & \na & \na & \na & \na & \na \\
		\hline
		\mathcal T^* & \mathbb N^* & \cellcolor{gray!25} \oeis{A086622}^{\dagger, \conj} & \na & \na & \na & \na & \na \\
		\hline \hline
		\mathcal R & \{1\} & \na & \Leftarrow & \na & \Leftarrow & \Leftarrow & \Leftarrow\\
		\hline
		\mathcal S & \{1\} & \cellcolor{gray!25} \oeis{A002212} & \Leftarrow & \cellcolor{gray!25} \oeis{A026376}^{\dagger, \conj} & \Leftarrow & \Leftarrow & \Leftarrow \\
		\hline
		\mathcal T & \{1\} & \cellcolor{gray!25} \substack{\oeis{A007317} \\ \oeis{A181768}^{\dagger}}& \Leftarrow & \na & \Leftarrow & \Leftarrow & \Leftarrow \\
		\hline
		\mathcal R^* & \{1\} & \na & \Leftarrow & \na & \Leftarrow & \Leftarrow & \Leftarrow \\
		\hline
		\mathcal S^* & \{1\} & \cellcolor{gray!25} \oeis{A090345}^{\conj} & \Leftarrow & \cellcolor{green!25} \oeis{A026571}^{\conj, \noa} & \Leftarrow & \Leftarrow & \Leftarrow \\
		\hline
		\mathcal T^* & \{1\} & \cellcolor{gray!25} \oeis{A090344}^{\dagger} & \Leftarrow & \na & \Leftarrow & \Leftarrow & \Leftarrow \\
		\hline \hline
		\mathcal R & \{2\} & \na & \Leftarrow & \na & \Leftarrow & \na & \Leftarrow \\
		\hline
		\mathcal S & \{2\} & \cellcolor{gray!25} \oeis{A085139} & \Leftarrow & \na & \Leftarrow & \na & \Leftarrow \\
		\hline
		\mathcal T & \{2\} & \cellcolor{gray!25} \oeis{A105864} ^{\conj}& \Leftarrow & \na & \Leftarrow & \na & \Leftarrow \\
		\hline
		\mathcal R^* & \{2\} & \na & \Leftarrow & \na & \Leftarrow & \na & \Leftarrow \\
		\hline
		\mathcal S^* & \{2\} & \na & \Leftarrow & \na & \Leftarrow & \na & \Leftarrow \\
		\hline
		\mathcal T^* & \{2\} & \na & \Leftarrow & \na & \Leftarrow & \na & \Leftarrow \\
		\hline \hline
		\mathcal R & \{1,2\} & \na & \na & \na & \na &  \na & \na \\
		\hline
		\mathcal S & \{1,2\} & \cellcolor{gray!25} \oeis{A084782} & \na & \na & \na &  \na & \na \\
		\hline
		\mathcal T & \{1,2\} & \cellcolor{gray!25} \oeis{A184018}^{\conj} & \na & \na & \na &  \na & \na \\
		\hline
		\mathcal R^* & \{1,2\} & \na & \na & \na & \na &  \na & \na \\
		\hline
		\mathcal S^* & \{1,2\} & \na & \na & \na & \na &  \na & \na \\
		\hline
		\mathcal T^* & \{1,2\} & \na & \na & \na & \na &  \na & \na \\
\hline
	\end{array} \\ 
	\\
	\textnormal{\normalsize Table A.1 for the case $\mathcal K = \mathbb N^*$}\\ \\
	\hline
	\end{array}
\end{align}
\normalsize
Then
\begin{align}
\label{eq:r:growthA254314}
[z^n]R(z) \sim \frac{\sqrt{5-32\zeta+46\zeta^2}}{2\sqrt{\zeta (1-2\zeta)^3  \pi n^3}}\cdot \zeta^{-n}.
\end{align}
\item \oeis{A025266}$^\dagger$.
For the class $\mathcal S^*$ again we have
\begin{align}
p(z) = \frac{4 z^2+ 2 z - 1}{2z - 1}
\end{align}
and thus we conclude that
\begin{align}
[z^n]S^*(z) \sim \frac{\sqrt{5+3\sqrt{5}}}{ \sqrt{(1+\sqrt{5})\pi n^3} }
\cdot \left(\frac{4}{\sqrt{5}-1}\right)^n .
\end{align}
\end{itemize}
If $\mathcal N = \{1\}$, then $F(z) = 1/(1-z)$ and in this situation we have one case:
\begin{itemize}
\item \oeis{A026571}. 
For the class $\mathcal S^*$ we have
\begin{align}
p(z) = \frac{1-z-4z^2}{1-z}
\end{align}
and hence for $\kappa$ either $\chi$ or $\xi$ (they coincide in this case), we have
\begin{align}
[z^n]\Omega(z) \sim \frac{16\sqrt{9-\sqrt{17}}}{ \sqrt{\pi(\sqrt{17}-1)^5\sqrt{17} n} }
\cdot \left(\frac{8}{\sqrt{17}-1}\right)^n ,
\end{align}
and since
\begin{align}
[z^n]S^*(z) \sim \frac{\sqrt{\sqrt{17}(\sqrt{17}-1)}}{ \sqrt{\pi(9-\sqrt{17}) n^3} }
\cdot \left(\frac{8}{\sqrt{17}-1}\right)^n ,
\end{align}
we conclude that
\begin{align}
\mathbb E_{\mathcal S^*}(\kappa) \sim \frac{16(9-\sqrt{17})}{\sqrt{17}(\sqrt{17}-1)^3}\cdot n
\end{align}
which is about $\% 62.13$ of $n$.
\end{itemize}
\end{example}

\subsection{Case $\mathcal K = \{2\}$.}
\label{sec:appendix:2}
Table A.2 in equation \eqref{eq:Tables:2} summarizes our findings for the case $\mathcal K = \{ 2 \}$. There were four cases among those that we found in OEIS that have no asymptotic description. Let us work these out:
\footnotesize
\begin{align}
	\label{eq:Tables:2}
	\begin{array}{|c|}
	\hline \\
	\begin{array}{|| c | c ||| c | c || c | c || c | c ||}
		\hline
		\textnormal{Class} & \mathcal N & \textnormal{roted plane} & \substack{\textnormal{non-}\\ \textnormal{plane}}
			& \substack{\#\textnormal{ petals}\\ \textnormal{plane}} 
			& \substack{\#\textnormal{ petals}\\ \textnormal{non-plane}}
			& \substack{\#\textnormal{ edges}\\ \textnormal{in petals}\\ \textnormal{plane}} 
			& \substack{\#\textnormal{ edges} \\ \textnormal{in petals}\\ \textnormal{non-plane}} \\
		\hline \hline \hline
		\mathcal R & \mathbb N^* & \cellcolor{green!25} \oeis{A173992}^{\conj,\noa} & \na & \na & \na & \na & \na \\
		\hline
		\mathcal S & \mathbb N^* & \na & \na & \na & \na & \na & \na \\
		\hline
		\mathcal T & \mathbb N^* & \cellcolor{gray!25} \oeis{A135052}^{\conj} & \na & \na & \na & \na & \na \\
		\hline
		\mathcal R^* & \mathbb N^* & \na & \na & \na & \na & \na & \na \\
		\hline
		\mathcal S^* & \mathbb N^* & \na & \na & \na & \na & \na & \na \\
		\hline
		\mathcal T^* & \mathbb N^* & \cellcolor{green!25} \oeis{A025276}^{\dagger,\noa} & \na & \na & \na & \na & \na \\
		\hline \hline
		\mathcal R & \{1\} & \cellcolor{gray!25} \oeis{A090344} & \Leftarrow & \na & \Leftarrow & \Leftarrow & \Leftarrow\\
		\hline
		\mathcal S & \{1\} & \cellcolor{gray!25} \oeis{A002026}^{\dagger} & \Leftarrow & \cellcolor{gray!25} 2\cdot\oeis{A002026}^{\dagger,\noa} & \Leftarrow & \Leftarrow & \Leftarrow \\
		\hline
		\mathcal T & \{1\} & \cellcolor{gray!25} \oeis{A001006} & \Leftarrow & \cellcolor{gray!25} \oeis{A005717} & \Leftarrow & \Leftarrow & \Leftarrow \\
		\hline
		\mathcal R^* & \{1\} & \cellcolor{gray!25} \oeis{A216604}^{\dagger} & \Leftarrow & \na & \Leftarrow & \Leftarrow & \Leftarrow \\
		\hline
		\mathcal S^* & \{1\} & \na & \Leftarrow & \na & \Leftarrow & \Leftarrow & \Leftarrow \\
		\hline
		\mathcal T^* & \{1\} & \cellcolor{gray!25} \oeis{A023426}^{\dagger,\conj} & \Leftarrow & \na & \Leftarrow & \Leftarrow & \Leftarrow \\
		\hline \hline
		\mathcal R & \{2\} & \cellcolor{gray!25} \oeis{A007317}^{\dagger} & \Leftarrow & \na & \Leftarrow & \na & \Leftarrow \\
		\hline
		\mathcal S & \{2\} & \cellcolor{gray!25} \oeis{A006319}^{\dagger} & \Leftarrow & \cellcolor{gray!25} 2\cdot\oeis{A026002}^{\dagger,\conj}  & \Leftarrow & \na & \Leftarrow \\
		\hline
		\mathcal T & \{2\} & \cellcolor{gray!25} \oeis{A006318}^{\dagger,\conj} & \Leftarrow & \cellcolor{gray!25} \oeis{A002002}^{\dagger} & \Leftarrow & \cellcolor{gray!25} \oeis{A110170}^{\dagger} & \Leftarrow \\
		\hline
		\mathcal R^* & \{2\} & \cellcolor{gray!25} \oeis{A090344}^{\dagger} & \Leftarrow & \na & \Leftarrow & \na & \Leftarrow \\
		\hline
		\mathcal S^* & \{2\} & \cellcolor{green!25} \oeis{A052702}^{\dagger,\noa} & \Leftarrow & \na & \Leftarrow & \na & \Leftarrow \\
		\hline
		\mathcal T^* & \{2\} & \cellcolor{green!25} \oeis{A023431}^{\dagger,\conj,\noa} & \Leftarrow & \na & \Leftarrow & \na & \Leftarrow \\
		\hline \hline
		\mathcal R & \{1,2\} & \cellcolor{gray!25} \oeis{A186334}^{\conj} & \na & \na & \na & \na & \na \\
		\hline
		\mathcal S & \{1,2\} & \na & \na & \na & \na & \na & \na \\
		\hline
		\mathcal T & \{1,2\} & \cellcolor{gray!25} \oeis{A128720} & \na & \na & \na & \na & \na \\
		\hline
		\mathcal R^* & \{1,2\} & \na & \na & \na & \na & \na & \na \\
		\hline
		\mathcal S^* & \{1,2\} & \na & \na & \na & \na & \na & \na \\
		\hline
		\mathcal T^* & \{1,2\} & \na & \na & \na & \na & \na & \na \\
\hline 
	\end{array} \\ 
	\\
	\textnormal{\normalsize Table A.2 for the case $\mathcal K = \{ 2 \}$} \\ \\
	\hline
	\end{array}
\end{align}
\normalsize

\begin{example} [$\mathcal K = \{2\}$: \oeis{A173992}, \oeis{A025276}$^\dagger$, \oeis{A052702}$^\dagger$, \oeis{A023431}$^\dagger$] Suppose that $\mathcal K= \{2\}$ (see Table in equation \eqref{eq:Tables:2}).
If $\mathcal N = \mathbb N^*$ and flowers are rooted-plane, then $F(z) = (1-z)/(1-2z)$ and hence we have the following cases:
\begin{itemize}
\item \oeis{A173992}. 
For the class $\mathcal R$, we have
\begin{align}
p(z) = \frac{4 z^3-4 z^2-2 z+1}{1-2 z}
\end{align}
and hence the dominant singularity is
\begin{align}
\zeta = \frac{1}{12} \left(\left(1+i \sqrt{3}\right) \sqrt[3]{1+3 i \sqrt{111}}+\frac{80}{\left(\sqrt{3}+i\right)^2 \sqrt[3]{1+3 i \sqrt{111}}}+4\right)\approx 0.3444\ldots .
\end{align}
Then
\begin{align}
[z^n]R(z) \sim \frac{\sqrt{2-5\zeta + 4\zeta^2}}{2\zeta(1-2\zeta)\sqrt{\pi n^3}} \cdot \zeta ^{-n}.
\end{align}
\item \oeis{A025276}$^\dagger$.
For the class $\mathcal T^*$, we have
\begin{align}
p(z)= \frac{(1-2z+2z^2)(1-2z-2z^2)}{(1-2z)^2}
\end{align}
and hence
\begin{align}
	[z^n]T^*(z) \sim \sqrt{\frac{2\sqrt{3}}{(3\sqrt{3}-5)\pi n^3}} \left(\frac{2}{\sqrt{3}-1}\right)^n .
\end{align}
\end{itemize}
If $\mathcal N = \{2\}$ and flowers are rooted-plane, then $F(z) = 1/(1-z^2)$ and we have the following cases:
\begin{itemize}
\item \oeis{A052702}$^\dagger$.
For the class $\mathcal S^*$, we have
\begin{align}
p(z)= \frac{(1-z^2+2z^3)(1-z^2-2z^2)}{(1-z)^2(1+z)^2}
\end{align}
and hence there are two dominant singularities, namely
\begin{align}
\zeta &= \frac{1}{6} \left(\sqrt[3]{6 \sqrt{78}+53}+\sqrt[3]{53-6 \sqrt{78}}-1\right) \approx 0.657298106\ldots 
\end{align}
and $-\zeta$. Thus
\begin{align}
	[z^{2n}]S^*(z) \sim
	 \sqrt{\frac{(3-\zeta^2)(1-\zeta^2)}{4\zeta^6 \pi n^3}} \zeta^{-2n} 
\quad \textnormal{ and } \quad [z^{2n+1}]S^*(z)=0.
\end{align}
\item\oeis{A023431}$^\dagger$.
For the class $\mathcal T^*$ the situation is very similar to the previous case for $\mathcal S^*$ (\oeis{A052702}$^\dagger$), we get
\begin{align}
	[z^{2n}]T^*(z) \sim
	\sqrt{\frac{(3-\zeta^2)(1-\zeta^2)}{4\zeta^2 \pi n^3}} \cdot \zeta^{-2n} 
	\quad \textnormal{ and } \quad
	[z^{2n+1}]T^*(z)=0.
\end{align}
\end{itemize}
\end{example}

\subsection{Case $\mathcal K = \{1,2\}$.}
\label{sec:appendix:3}

Table A.3 in equation \eqref{eq:Tables:12} summarizes our findings for the case $\mathcal K = \{ 1, 2 \}$. There were three cases among those that we found in OEIS that have no asymptotic description. Let us work these out:

\footnotesize
\begin{align}
	\label{eq:Tables:12}
	\begin{array}{|c|}
	\hline
	\\
	\begin{array}{|| c | c ||| c | c || c | c || c | c ||}
		\hline
		\textnormal{Class} & \mathcal N & \textnormal{rooted-plane} & \substack{\textnormal{non-}\\ \textnormal{plane}}
			& \substack{\#\textnormal{ petals}\\ \textnormal{plane}} 
			& \substack{\#\textnormal{ petals}\\ \textnormal{non-plane}}
			& \substack{\#\textnormal{ edges}\\ \textnormal{in petals}\\ \textnormal{plane}} 
			& \substack{\#\textnormal{ edges} \\ \textnormal{in petals}\\ \textnormal{non-plane}} \\
		\hline \hline \hline
		\mathcal R & \mathbb N^* & \na & \na & \na & \na & \na & \na \\
		\hline
		\mathcal S & \mathbb N^* & \na & \na & \na & \na & \na & \na \\
		\hline
		\mathcal T & \mathbb N^* & \na & \na & \na & \na & \na & \na \\
		\hline
		\mathcal R^* & \mathbb N^* & \na & \na & \na & \na & \na & \na \\
		\hline
		\mathcal S^* & \mathbb N^* & \cellcolor{green!25} \oeis{A026134}^{\dagger,\noa} & \na & \na & \na & \na & \na \\
		\hline
		\mathcal T^* & \mathbb N^* & \cellcolor{gray!25} \oeis{A002026} & \na & \cellcolor{green!25} \oeis{A014531}^{\dagger,\noa} & \na & \cellcolor{gray!25} \oeis{A097894}^{\dagger,\conj} & \na \\
		\hline \hline
		\mathcal R & \{1\} & \cellcolor{gray!25} \oeis{A105633} & \Leftarrow & \na & \Leftarrow & \Leftarrow & \Leftarrow\\
		\hline
		\mathcal S & \{1\} & \cellcolor{gray!25} \oeis{A071724} & \Leftarrow & \cellcolor{gray!25} \oeis{A220101}^{\dagger} & \Leftarrow & \Leftarrow & \Leftarrow \\
		\hline
		\mathcal T & \{1\} & \cellcolor{gray!25} \oeis{A000108}^{\dagger} & \Leftarrow & \cellcolor{gray!25} \oeis{A001791} & \Leftarrow & \Leftarrow & \Leftarrow \\
		\hline
		\mathcal R^* & \{1\} & \na & \Leftarrow & \na & \Leftarrow & \Leftarrow & \Leftarrow \\
		\hline
		\mathcal S^* & \{1\} & \na & \Leftarrow & \na & \Leftarrow & \Leftarrow & \Leftarrow \\
		\hline
		\mathcal T^* & \{1\} & \cellcolor{gray!25} \oeis{A026418}^{\dagger,\conj} & \Leftarrow & \na & \Leftarrow & \Leftarrow & \Leftarrow \\
		\hline \hline
		\mathcal R & \{2\} & \na & \Leftarrow & \na & \Leftarrow & \na & \Leftarrow \\
		\hline
		\mathcal S & \{2\} & \na & \Leftarrow & \na & \Leftarrow & \na & \Leftarrow \\
		\hline
		\mathcal T & \{2\} & \cellcolor{gray!25} \oeis{A128720} & \Leftarrow & \cellcolor{gray!25} \oeis{A106053}^{\noa} & \Leftarrow & \cellcolor{green!25} 2\cdot\oeis{A106053}^{\noa}& \Leftarrow \\
		\hline
		\mathcal R^* & \{2\} & \na & \Leftarrow & \na & \Leftarrow & \na & \Leftarrow \\
		\hline
		\mathcal S^* & \{2\} & \na  & \Leftarrow & \na & \Leftarrow & \na & \Leftarrow \\
		\hline
		\mathcal T^* & \{2\} & \na  & \Leftarrow & \na & \Leftarrow & \na & \Leftarrow \\
		\hline \hline
		\mathcal R & \{1,2\} & \na & \na & \na & \na & \na & \na \\
		\hline
		\mathcal S & \{1,2\} & \na & \na & \na & \na & \na & \na \\
		\hline
		\mathcal T & \{1,2\} & \cellcolor{gray!25} \oeis{A085139}^{\dagger} & \na & \na & \na & \na & \na \\
		\hline
		\mathcal R^* & \{1,2\} & \na & \na & \na & \na & \na & \na \\
		\hline
		\mathcal S^* & \{1,2\} & \na & \na & \na & \na & \na & \na \\
		\hline
		\mathcal T^* & \{1,2\} & \na & \na & \na & \na & \na & \na \\
\hline 
	\end{array} \\ 
	\\
	\textnormal{\normalsize Table A.3 for the case $\mathcal K = \{ 1,2 \}$} \\  \\
	\hline
	\end{array}
\end{align}
\normalsize

\begin{example} [$\mathcal K = \{1,2\}$: \oeis{A026134}$^\dagger$, \oeis{A014531}$^\dagger$, \oeis{A106053}] Suppose that $\mathcal K= \{1,2\}$.
If $\mathcal N = \mathbb N^*$ and flowers are rooted-plane, then $F(z) = (1-z)/(1-2z)$ and hence we have the following cases:
\begin{itemize}
\item \oeis{A026134}$^\dagger$.
For the class $\mathcal S^*$, we have
\begin{align}
p(z)= \frac{(1-z)^2(1+z)(1-3z)}{(1-2z)^2}
\end{align}
and hence
\begin{align}
	[z^n]S^*(z) \sim \sqrt{\frac{27}{\pi n^3}} 3^n .
\end{align}
\item \oeis{A014531}$^\dagger$.
For the class $\mathcal T^*$ and $\chi\colon\mathcal T^* \to \mathbb N$
the number of petals, the situation is very similar to the previous case for $\mathcal S^*$, we get
\begin{align}
	[z^n]\Omega(z) \sim \sqrt{\frac{27}{4 \pi n}} 3^n 
\end{align}
and since
\begin{align}
	[z^n]T^*(z) \sim \sqrt{\frac{27}{ \pi n^3}} 3^n ,
\end{align}
we conclude that
\begin{align}
	\mathbb E_{\mathcal T_n^*}(\xi) \sim \frac{1}{2}n.
\end{align}
\end{itemize}
If $\mathcal N = \{2\}$, then $F(z)=1/(1-z^2)$ and hence we have the following case:
\begin{itemize}
\item $2\times$\oeis{A106053}.
For the class $\mathcal T$ and $\xi\colon\mathcal T^* \to \mathbb N$
the number of edges in petals, we have
\begin{align}
	p(z) = \frac{(1+z-z^2)(1-3z-z^2)}{(1-z)^2(1+z)^2}
\end{align}
and then
\begin{align}
	[z^n]\Omega(z) \sim \frac{2\sqrt{2}}{(3+\sqrt{13})\sqrt{(11\sqrt{13}-39)\pi n}} \left(\frac{2}{\sqrt{13}-3}\right)^n,
\end{align}
and analysis as above in particular yield that $\mathbb E_{\mathcal T_n}(\xi)$ is about $16.79\%$ of $n$.
\end{itemize}
\end{example}

\subsection{Case $\mathcal K = \{ 1 \}$.}
\label{sec:appendix:4}
\footnotesize
\begin{align}
\footnotesize
	\label{eq:Tables:1}
	\begin{array}{|c|}
	\hline \\
	\begin{array}{|| c | c ||| c | c || c | c || c | c ||}
		\hline
		\textnormal{Class} & \mathcal N & \textnormal{rooted-plane} & \substack{\textnormal{non-plane}}
			& \substack{\#\textnormal{ petals}\\ \textnormal{rooted-plane}} 
			& \substack{\#\textnormal{ petals}\\ \textnormal{non-plane}}
			& \substack{\#\textnormal{ edges}\\ \textnormal{in petals}\\ \textnormal{rooted-plane}} 
			& \substack{\#\textnormal{ edges} \\ \textnormal{in petals}\\ \textnormal{non-plane}} \\
		\hline \hline \hline
		\mathcal R & \mathbb N^* & \cellcolor{gray!25} \oeis{A000079} & \cellcolor{gray!25} \oeis{A000070} & \cellcolor{gray!25} \oeis{A001787} & \cellcolor{gray!25} \oeis{A284870} & \cellcolor{gray!25} \oeis{A000337} & \cellcolor{gray!25} \oeis{A182738}^{\dagger} \\
		\hline
		\mathcal S & \mathbb N^* & \cellcolor{gray!25} \oeis{A001519} &\cellcolor{gray!25} \oeis{A067687} & \cellcolor{gray!25} \oeis{A001870}^{\dagger} & \na & \cellcolor{gray!25} \oeis{A001871}^{\dagger} & \na \\
		\hline
		\mathcal T & \mathbb N^* & \cellcolor{gray!25} \oeis{A001519}^{\dagger} & \cellcolor{gray!25} \oeis{A067687}^{\dagger} & \cellcolor{gray!25} \oeis{A001870}^{\dagger} & \na & \cellcolor{gray!25} \oeis{A001871}^{\dagger} & \na\\
		\hline
		\mathcal R^* & \mathbb N^* & \cellcolor{gray!25} \oeis{A000225} & \cellcolor{gray!25} \oeis{A026905}^{\dagger} &\cellcolor{gray!25} \oeis{A001787} & \cellcolor{gray!25} \oeis{A284870} & \cellcolor{gray!25} \oeis{A000337} & \cellcolor{gray!25} \oeis{A182738}^{\dagger} \\
		\hline
		\mathcal S^* & \mathbb N^* & \cellcolor{gray!25} \oeis{A000129}^{\dagger} & \na & \cellcolor{gray!25} \oeis{A026937}^{\dagger} & \na & \cellcolor{gray!25} \oeis{A006645} & \na \\
		\hline
		\mathcal T^* & \mathbb N^* & \cellcolor{gray!25} \oeis{A000129}^{\dagger} & \na & \cellcolor{gray!25} \oeis{A026937}^{\dagger} & \na & \cellcolor{gray!25} \oeis{A006645}^{\dagger} & \na \\
		\hline \hline
		\mathcal R & \{1\} & \cellcolor{gray!25} \oeis{A000027} & \Leftarrow & \cellcolor{gray!25} \oeis{A000217}  & \Leftarrow & \Leftarrow & \Leftarrow\\
		\hline
		\mathcal S & \{1\} & \cellcolor{gray!25} \substack{\oeis{A011782} \\ \oeis{A000079}^{\dagger}}& \Leftarrow & \cellcolor{gray!25} \oeis{A139756} & \Leftarrow & \Leftarrow & \Leftarrow \\
		\hline
		\mathcal T & \{1\} & \cellcolor{gray!25} \oeis{A000079} & \Leftarrow & \cellcolor{gray!25} \substack{\oeis{A001787} \\ \cellcolor{gray!25} \oeis{A139756}}^{\dagger}& \Leftarrow & \Leftarrow & \Leftarrow \\
		\hline
		\mathcal R^* & \{1\} & \cellcolor{gray!25} \oeis{A001477} & \Leftarrow & \cellcolor{gray!25} \oeis{A000217} & \Leftarrow & \Leftarrow & \Leftarrow \\
		\hline
		\mathcal S^* & \{1\} & \cellcolor{gray!25} \substack{\oeis{A212804}\\ \oeis{A000045}^{\dagger}}& \Leftarrow & \cellcolor{gray!25} \oeis{A001629} & \Leftarrow & \Leftarrow & \Leftarrow \\
		\hline
		\mathcal T^* & \{1\} & \cellcolor{gray!25} \oeis{A000045} & \Leftarrow & \cellcolor{gray!25} \oeis{A001629}^{\dagger} & \Leftarrow & \Leftarrow & \Leftarrow \\
		\hline \hline
		\mathcal R & \{2\} & \cellcolor{gray!25} \oeis{A008619} & \Leftarrow & \cellcolor{gray!25} \oeis{A000217}^{\dagger} & \Leftarrow & \cellcolor{gray!25} \oeis{A002378}^{\dagger}& \Leftarrow \\
		\hline
		\mathcal S & \{2\} & \cellcolor{gray!25} \substack{\oeis{A324969} \\ \cellcolor{gray!25} \oeis{A000045}^{\dagger} }& \Leftarrow & \cellcolor{gray!25} \oeis{A001269}^{\dagger} & \Leftarrow & \cellcolor{gray!25} 2\cdot\oeis{A001629}^{\dagger}  & \Leftarrow \\
		\hline
		\mathcal T & \{2\} &  \cellcolor{gray!25} \oeis{A000045}^{\dagger} & \Leftarrow & \cellcolor{gray!25} \oeis{A001269} & \Leftarrow & \cellcolor{gray!25} 2\cdot\oeis{A001629}^{\dagger} & \Leftarrow \\
		\hline
		\mathcal R^* & \{2\} & \cellcolor{gray!25} \oeis{A004526} & \Leftarrow & \cellcolor{gray!25} \oeis{A000217}^{\dagger} & \Leftarrow & \cellcolor{gray!25} \oeis{A002378}^{\dagger} & \Leftarrow \\
		\hline
		\mathcal S^* & \{2\} & \cellcolor{gray!25} \oeis{A000931} & \Leftarrow & \cellcolor{gray!25} \oeis{A228577}^{\dagger} & \Leftarrow & \cellcolor{green!25} 2\cdot\oeis{A228577}^{\dagger,\noa} & \Leftarrow \\
		\hline
		\mathcal T^* & \{2\} & \cellcolor{gray!25} \oeis{A000931}^{\dagger} & \Leftarrow & \cellcolor{gray!25} \oeis{A228577}^{\dagger} & \Leftarrow & \cellcolor{green!25} 2\cdot\oeis{A228577}^{\dagger,\noa} & \Leftarrow \\
		\hline \hline
		\mathcal R & \{1,2\} & \cellcolor{gray!25} \oeis{A000071}^{\dagger} & \cellcolor{gray!25} \oeis{A087811}^{\circ} & \cellcolor{green!25} \oeis{A002940}^{\dagger,\noa} & \cellcolor{gray!25} \oeis{A111384}^{\dagger} & \cellcolor{gray!25} \oeis{A094584}^{\dagger} &  \na  \\
		\hline
		\mathcal S & \{1,2\} & \cellcolor{gray!25} \substack{\oeis{A215928} \\ \cellcolor{gray!25} \oeis{A000129}^{\dagger}}& \cellcolor{gray!25} \oeis{A106805}^{\dagger,\noa} & \na & \na &  \na &  \na \\
		\hline
		\mathcal T & \{1,2\} & \cellcolor{gray!25} \oeis{A000129}^{\dagger} & \cellcolor{gray!25} \oeis{A106805}^{\noa} & \na & \na & \na &  \na  \\
		\hline
		\mathcal R^* & \{1,2\} & \cellcolor{gray!25} \oeis{A001911} & \cellcolor{gray!25} \oeis{A024206}^{\conj} & \cellcolor{green!25} \oeis{A002940}^{\dagger,\noa} & \cellcolor{gray!25} \oeis{A111384}^{\dagger} & \cellcolor{gray!25} \oeis{A094584}^{\dagger} & \na \\
		\hline
		\mathcal S^* & \{1,2\} & \cellcolor{green!25}  \substack{\oeis{A142474}^{\noa} \\ \oeis{A141015}^{\dagger,\noa}} & \cellcolor{green!25} \oeis{A108742}^{\dagger,\noa} & \na & \na & \na & \na  \\
		\hline
		\mathcal T^* & \{1,2\} & \cellcolor{green!25} \oeis{A141015}^{\noa} & \cellcolor{green!25} \oeis{A108742}^{\dagger,\noa}  & \na & \na & \na &  \na  \\
		\hline
	\end{array} \\
	\\
	\textnormal{\normalsize Table A.4 for the case $\mathcal K = \{ 1 \}$}  \\ \\ \hline
	\end{array}
\end{align}

\normalsize

\begin{example} [$\mathcal K =\{1\}$: $2\times$\oeis{A228577}$^\dagger$, \oeis{A002940}$^\dagger$, \oeis{A142474}, \oeis{A108742}$^\dagger$] Suppose that $\mathcal K= \{1\}$ (see Table in equation \eqref{eq:Tables:1}). If $\mathcal N= \{2\}$, then $F(z)=1/(1-z^2)$ and we have the following case:
\begin{itemize}
\item $2\times$\oeis{A228577}$^\dagger$.
For the class $\mathcal S^*$ and $\xi$ the parameter that returns the number of edges in all the petals of the tree with flowers, $F(z,u)=1/(1-u^2z^2)$,
hence $\alpha$ is the positive root of $1-z^2-z^3$, that is,
\begin{align}
\alpha = \frac{1}{6} \left(2^{2/3} \sqrt[3]{3 \sqrt{69}+25}+2^{2/3} \sqrt[3]{25-3 \sqrt{69}}-2\right)
\end{align}
and since $r=2$,
\begin{align}
[z^n]\Omega(z)\sim\frac{(1-\alpha^2)^2n}{2\alpha^3}\alpha^{-n}
\end{align}
and similar analysis in particular yield that $\mathbb E_{\mathcal S_n^*}(\xi)$ is about $82.29\%$ of $n$.
\end{itemize}
If $\mathcal N= \{1,2\}$ we have the following cases:
\begin{itemize}
\item \oeis{A002940}$^\dagger$.
If flowers are rooted-plane, then $F(z)=1/(1-z-z^2)$.
In this case, for the class $\mathcal R$ and $\chi\colon \mathcal R \to \mathbb N$ the parameter that returns the number of petals,
\begin{align}
[z^n]\Omega(z)\sim\frac{(7+3\sqrt{5})n}{10}\left(\frac{2}{\sqrt{5}-1}\right)^n,
\end{align}
and one concludes that $\mathbb E_{\mathcal R_n}(\chi) \sim \frac{1}{2}(1+1/\sqrt{5})n$, about $72.36\%$.
\item \oeis{A142474}.
Again, if flowers are rooted-plane, then $F(z)=1/(1-z-z^2)$. 
In this case, for the class $\mathcal S^*$ we have that $\alpha$ is the real positive
root of $1-z-2z^2-z^3$, that is,
\begin{align}
\label{eq:alphaA142474}
\alpha = \frac{1}{6} \left(2^{2/3} \sqrt[3]{3 \sqrt{93}+29}+2^{2/3} \sqrt[3]{29-3 \sqrt{93}}-4\right)
\end{align}
and hence
\begin{align}
[z^n]f(z) = \frac{(1-\alpha-\alpha^2)^2}{\alpha^2(2+2\alpha-2\alpha^2-\alpha^3)}\alpha^{-n}.
\end{align}
\item \oeis{A108742}$^\dagger$.
If flowers are non-plane, then $F(z)=(1-z)^{-1}(1-z^2)^{-1}$.
In this case, for the class
$\mathcal S^*$, the dominant singularity $\alpha=0.5248885986\ldots$
is the real positive root of $1-z-2z^2+z^4$ and
\begin{align}
[z^n]f(z)\sim\frac{(1-\alpha)^3(1+\alpha)^2}{\alpha^2(2+4\alpha-2\alpha^2-\alpha^3+\alpha^4) }\alpha^{-n}.
\end{align}

\end{itemize}

\end{example}

\end{appendices}


\section*{Declarations}

{\bf Funding.} This work was supported by DGAPA-PAPIIT grants IN107718 and IN110221.

\bigskip

\noindent
{\bf Conflict of interest.} No potential conflict of interest was reported by the author.

\bibliography{GomezTWF-bibliography}

\end{document}